\documentclass[11pt]{amsart}

\usepackage{amssymb,amsmath,amsthm,amsfonts}
\usepackage{color}
\usepackage{graphicx}
\usepackage{soul}
\usepackage{comment}

\newcommand{\bea}{\begin{eqnarray}}
\newcommand{\eea}{\end{eqnarray}}

\usepackage[vcentermath]{youngtab}

\newcommand{\g}{\mathfrak g}

\usepackage[cp1250]{inputenc}
\usepackage[T1]{fontenc}
\usepackage{verbatim}

\usepackage[margin=2cm]{geometry} 
\newtheorem{theorem}{Theorem}
\newtheorem{lemma}[theorem]{Lemma}
\newtheorem{proposition}[theorem]{Proposition}
\newtheorem{corollary}[theorem]{Corollary}

\newtheorem{remark}[theorem]{Remark}

\numberwithin{theorem}{section}
\allowdisplaybreaks

\makeatletter \@addtoreset{equation}{section}

\makeatother \makeatletter

\begin{document}

\title[]{  On the representation theory of  the   vertex algebra $L_{-5/2}(sl(4))$ }

\author[]{Dra\v zen  Adamovi\' c}
\address{Department of Mathematics, Faculty of Science \\
	University of Zagreb \\
	Bijeni\v cka 30 \\
    Croatia }

\email{adamovic@math.hr}

\author[]{Ozren Per\v se}
\address{Department of Mathematics, Faculty of Science \\
	University of Zagreb \\
	Bijeni\v cka 30 \\
	Croatia
}

\email{perse@math.hr}

\author[]{Ivana Vukorepa}
\address{Department of Mathematics, Faculty of Science \\
	University of Zagreb \\
	Bijeni\v cka 30 \\
	Croatia
}
\email{vukorepa@math.hr} 

\begin{abstract}
	We  study  the representation theory of non-admissible simple affine vertex algebra $L_{-5/2} (sl(4))$. We determine an explicit formula for the singular vector of conformal weight four in the universal affine vertex algebra
	$V^{-5/2} (sl(4))$, and show that it generates the maximal ideal in $V^{-5/2} (sl(4))$. We classify irreducible $L_{-5/2} (sl(4))$--modules in the category ${\mathcal O}$, and determine the fusion rules between irreducible modules in the category of ordinary modules $KL_{-5/2}$. It turns out that this fusion algebra is isomorphic to the fusion algebra of $KL_{-1}$.	We also prove that $KL_{-5/2}$ is a semi-simple, rigid braided tensor category.
	
	 In our proofs we use the notion of collapsing level for the affine $\mathcal{W}$--algebra, and the properties of conformal embedding $gl(4)  \hookrightarrow sl(5)$ at level $k=-5/2$ from \cite{AKMPP-16}.
	 We show that $k=-5/2$ is a collapsing level with respect to the subregular nilpotent element $f_{subreg}$, meaning that the simple quotient of the affine $\mathcal{W}$--algebra $W^{-5/2}(sl(4), f_{subreg})$ is isomorphic to the Heisenberg vertex algebra $M_J(1)$. We prove certain results on vanishing and non-vanishing of cohomology for the quantum Hamiltonian reduction functor $H_{f_{subreg}}$. It turns out that the properties of $H_{f_{subreg}}$ are more subtle than in the case of minimal reducition.
	  
 \end{abstract}
 \maketitle

\section{Introduction}

Representation theory of simple affine vertex algebra $L_{k}( \mathfrak{g})$, for arbitrary simple Lie algebra $\mathfrak{g}$ and general level $k \in \mathbb{C}$ is a very interesting problem for many years. Some of the best understood cases are non-negative integer levels $k \in \mathbb{Z}_{\geq 0}$ (cf. \cite{FZ}, \cite{Li-local}), and a class of cetrain rational levels, called admissible levels (cf. \cite{AM}, \cite{Ara}, \cite{KW1}, \cite{KW2}), which includes non-negative integers. On the other hand, to the best of our knowledge, there are not many results on the representation theory of $L_{k}( \mathfrak{g})$ for non-admissible (non-generic, non-critical) levels $k \in \mathbb{Q}$. Most of the results obtained thus far are for negative integer levels, since they naturally appear in free-field realizations of certain simple affine vertex algebras (cf. \cite{AP-JAA}),
in the context of affine vertex algebras associated to the Deligne exceptional series (cf. \cite{ArMo-joseph}), and in the context of collapsing levels for minimal affine $\mathcal{W}$--algebras (cf. \cite{AKMPP-IMRN}).

But still there are many open cases where the representation theory  of affine vertex algebras is unknown. Let us mention here the  cases which, in our opinion,  will  play an important role in future analysis of non-admissible affine vertex algebras:
\begin{itemize}
\item[(1)] The affine vertex algebras at negative levels which appeared in the decompositions of conformal embeddings in \cite{AKMPP-16}, \cite{AKMPP-JJM}, \cite{AMPP-AIM}.
    \item[(2)] The affine vertex algebras at collapsing levels for affine $\mathcal{W}$--algebras.
    \end{itemize}
 Based on a recent work   on vertex tensor categories (cf.   \cite{C-vir-ten}, \cite{CY}), one expects that for all vertex algebras appearing in (1) and (2), the category $KL_k$ of ordinary modules will have the structure of a rigid braided tensor category. In the cases of collapsing levels for minimal affine $\mathcal W$--algebras, the paper \cite{AKMPP-IMRN} implies that the category $KL_k$ is semisimple. Using this result,   and the decompositions from (1), the paper \cite{CY}  shows that $KL_k$ has the structure of a rigid braided tensor category. But there are many   cases of non-admissible affine vertex algebras which appeared in (1), and which haven't been discussed  neither   in \cite{AKMPP-IMRN} or in \cite{CY}.  For such vertex algebras, the minimal affine $\mathcal{W}$--algebras are complicated, so for proving the existence of the vertex tensor category structure, one needs to apply different method. It seems that it is natural to investigate quantum Hamiltonian reduction and $\mathcal W_k(\g,f)$ for  non-minimal nilpotent element~$f$. Unfortunately,   the results on vanishing or non-vanishing of  QHR  functor $H_f(\, \cdot \,)$ are not presented so explicitly as in the case of the minimal reduction, so we can not easily generalize methods of \cite{AKMPP-IMRN}. Certain results in this direction have been obtained recently in  \cite{ArE}, \cite{ArEM}, but only in the cases of admissible levels.

In the present paper, we present a case study of a new  example of a   non-admissible affine vertex algebra which belongs to both cases mentioned above.
In particular, we consider the simple affine vertex algebra $L_{-5/2} (sl(4))$. This case is of particular interest to us, since looking at the low rank Lie algebras $\mathfrak{g} = sl(n)$, and using the criterion from \cite{GK}, one obtains that the level $k = -5/2$ for $\widehat{sl(4)}$ is the new example of non-admissible, non-generic half-integer level, which has not been investigated before. Another motivation for studying this case is the conformal embedding of $L_{-5/2} (sl(4)) \otimes M  (1)$ into 
$L_{-5/2} (sl(5))$ (cf. \cite{AKMPP-16}) (here $M (1)$ denotes the Heisenberg vertex algebra associated to abelian Lie algebra of rank one). 

Our first step in understanding the simple vertex algebra $L_{-5/2} (sl(4))$ is determining an explicit formula for a singular vector $v$ of conformal weight $4$ in the universal affine vertex algebra $V^{-5/2} (sl(4))$. This enables us to classify irreducible modules in the category $\mathcal{O}$ for the associated quotient 
$\widetilde{L}_{-5/2} (sl(4)) =V^{-5/2} (sl(4)) / \langle v \rangle$, using the methods from \cite{A-PhD, AM}. The irreducible representations are
$$ L_{-5/2}(\mu_i(t)), \ t\in \mathbb{C}, \ i=1, \dots, 16, $$
 and they are parametrized as a union of $16$ lines in ${\Bbb C}^3$ (cf. Theorem  \ref{kvocijent-moduli}). So there are uncountably  many irreducible modules in the category $\mathcal{O}$. A particularly interesting consequence is that the irreducible representations in the  category $KL_{-5/2}$ of ordinary modules are   $$ \pi_r = L_{-5/2}(r \omega_1),  \ \pi_{-r}= L_{-5/2} (r \omega_3) \quad (r \in {\Bbb Z}_{\ge 0}).$$

Our next goal is to conclude that the vertex algebra $\widetilde{L}_{-5/2} (sl(4))$ is simple, i.e. that the singular vector $v$ generates the maximal ideal in $V^{-5/2} (sl(4))$, what also happens in the case of admissible affine vertex  algebras (cf. \cite{AM}, \cite{Ara}). Instead of proving this claim directly in the affine vertex algebra setting, we use affine $\mathcal{W}$-algebras. A close related  problem is to prove the semi-simplicity of the category of ordinary $L_{-5/2} (sl(4))$--modules.

Similar situation was in the case of affine vertex algebras $V^{-1}(sl(n))$, for $n \ge 3$. The classification of irreducible $L_{-1}(sl(n))$--modules was obtained in \cite{AP-2007, AP-JAA}, by using very similar methods to those which we use in the current paper (only formulas for singular vectors were much simpler).  The description of maximal ideal in $V^{-1}(sl(n))$ was obtained in \cite{ArMo-sheets}, and the semi-simplicity of the category $KL_{-1}$ was proved in \cite{AKMPP-IMRN}. For a review of this approach see Subsection \ref{review-1}. 

We show that the similar approach can be achieved in the case of affine vertex algebra $V^{-5/2}(sl(4))$. First we consider affine $\mathcal{W}$-algebra $W^{-5/2}(sl(4), f_{subreg})$, and show that its simple quotient is isomorphic to the Heisenberg vertex algebra. Then, the simplicity of $\widetilde{L}_{-5/2} (sl(4))$, and 
the semi-simplicity of $KL_{-5/2}$ is proved using the properties of the quantum Hamiltonian reduction functor  $H_{f_{subreg}}$. It turns out that $H_{f_{subreg}}$ has slightly different properties than in the case of the minimal reduction.

 We prove in  Theorem \ref{conj-p}:

\begin{theorem}  $H_{f_{subreg}} (\pi_r) \ne  \{0\}$ for $r\le 0$, but $H_{f_{subreg}} (M) =\{0\}$ for any highest weight module in $KL_{-5/2}$ of weight $n \omega_1$, $ n >0$.  
\end{theorem}
 
 It seems that  in general the non-vanishing of QHR functor for irreducible modules can be determined by the methods from \cite{Ara-11}. Since our proof requires more precise information  for modules which are not neccasary irreducible, we decide to present a different   proof.
 The proof    is given  in Section  \ref{singular-new} and it 
is based on the following key observations:

\begin{itemize}
\item Let $k=-5/2$. The generalized Verma module $V^k(\mu)$ for $\mu=n\omega_1, n\omega_3$, $n >0$, contains a singular vector $w_{\nu}$ of degree $2$, and $w_{\nu}$ belongs to the submodule $J^k \cdot V^{k}(\mu)$ (here $J^k$ denotes the ideal in $V^{-5/2} (sl(4))$ generated by singular vector $v$). Then we construct universal $\widetilde L_k(sl(4))$--modules $\overline M(\mu)$.

\item For $\mu = n \omega_1$, the QHR functor $H_{f_{subreg}}$ sends $w_{\nu}$ to a highest weight vector of $H_{f_{subreg}} (V^{k}(\mu))$, implying the vanishing of cohomology for all highest weight modules in $KL_k$ of the weight $\mu$.

\item For $\mu = n \omega_3$, $w_{\nu}$ is mapped to a singular vector  of degree one in  $H_{f_{subreg}} (V^{k}(\mu))$, which generates a proper submodule. Using this we show that $H_{f_{subreg}} (\overline M(n \omega_3)) \ne \{0\}$. The classification of irreducible $\widetilde L_k(sl(4))$--modules enables us to show  non-vanishing of cohomology   for all highest weight modules in $KL_k$ of the weight $\mu$.
\end{itemize}

 We show in Theorem  \ref{main-max}   how these properties imply that $\widetilde{L}_{-5/2} (sl(4))$ is simple and $KL_{-5/2}$ is  semi-simple.
 
Next, we exploit the conformal embedding of $L_{-5/2} (sl(4)) \otimes M  (1)$ into $L_{-5/2} (sl(5))$ and the result from \cite{AKMPP-16} which says that 
all modules $\pi_r$, $r \in {\Bbb Z}$ are realized inside of $L_{-5/2} (sl(5))$. 
Using that fact, we are able to determine the fusion rules for irreducible $L_{-5/2} (sl(4))$--modules in the category $KL_{-5/2}$ (cf.  Proposition   \ref{fusion-rules}):
\bea \label{fr-uvod} \pi_r \times \pi_s = \pi_{r+s} \quad (r,s \in {\Bbb Z}).  \eea
Then applying results of the new work \cite{CY}, we prove that $KL_{-5/2}$ is a rigid, braided tensor category. 

Let us summarize main results of our paper in the following theorem:
\begin{theorem}Let $k=-5/2$.
\item[(1)] The set $\{ \pi_r \ \vert \  r \in {\Bbb Z}\}$ provides a complete list of irreducible $L_{k}(sl(4))$--modules in the category $KL_{k}$.
\item[(2)] The maximal ideal $J^k$  of $V^k(sl(4))$ is generated by  a singular vector $v$ of degree $4$ and $\g$--weight $2 \omega_2$ (cf. Theorem \ref{thm-singv}).
\item[(3)] $KL_{-5/2}$ is a semi-simple, rigid  braided tensor category with the  fusion rules (\ref{fr-uvod}).
\item[(4)] The set $\{ L_{-5/2} (\mu_i(t)) \ \vert \, i=1, \dots, 16, t \in {\Bbb C}\}$ provides a complete list of irreducible $L_{-5/2}(sl(4))$--modules in the category $\mathcal O$.
\item[(5)] There exist indecomposable $L_{-5/2}(sl(4))$--modules in the category $\mathcal O$ (cf. Remark \ref{remark-ind}).
\end{theorem}

We observe that the top components of irreducible $L_{-5/2} (sl(4))$--modules from the category $KL_{-5/2}$ are the same as the top components of irreducible $L_{-1} (sl(4))$--modules from the category $KL_{-1}$, and that the fusion rules for irreducible modules from $KL_{-5/2}$ and $KL_{-1}$  coincide (cf. \cite{AP-JAA}).

We also determine the decompositions of irreducible $L_{-5/2} (sl(5))$--modules in the category $\mathcal{O}$ (note that the level $k = -5/2$ is admissible for $\widehat{sl(5)}$) as $L_{-5/2} (sl(4)) \otimes M  (1)$--modules. 
Furthermore, we show that there is a homomorphism from Zhu's algebra of 
$L_{-5/2} (sl(4))$ to associated Weyl algebra and study certain  $L_{-5/2}( sl(4))$--modules as subquotients of relaxed $L_{-5/2}( sl(5))$--modules.

\vskip 5mm
{\bf Acknowledgements.}
The authors  are  partially supported   by the
QuantiXLie Centre of Excellence, a project cofinanced
by the Croatian Government and European Union
through the European Regional Development Fund - the
Competitiveness and Cohesion Operational Programme
(KK.01.1.1.01.0004).

\section{Preliminaries} \label{sect-prelim}

We assume that the reader is familiar with the notions of vertex algebra (cf. \cite{Bo}, \cite{FHL}, \cite{FLM}, \cite{K2}), simple Lie algebra and affine Kac-Moody algebra (cf. \cite{K1}). In this section we recall some results on affine vertex algebras. We also discuss affine $\mathcal{W}$--algebras following the papers \cite{ArMo-joseph}, \cite{ArMo-sheets}, \cite{KW3}.

\subsection{Affine vertex algebras}
Let $\mathfrak{g}$ be a simple Lie algebra with a triangular decomposition 
$\mathfrak{g} = \mathfrak{n}_{-} \oplus \mathfrak{h} \oplus \mathfrak{n}_{+}$.
Denote by $\left( \cdot \, | \, \cdot \right)$ the invariant bilinear form on $\mathfrak{g}$, normalized by the condition $\left( \theta \, | \, \theta \right)=2$, where $\theta$ is the highest root of $\mathfrak{g}$. We will often identify $\mathfrak{h}^*$ with $\mathfrak{h}$ via $\left( \cdot \, | \, \cdot \right)$. 
For $\mu \in \mathfrak{h}^*$, denote by $V(\mu )$ the irreducible highest weight $\mathfrak{g}$--module with highest weight $\mu$. Let us denote by $\alpha_1, \ldots, \alpha_{\ell}$ simple roots, by $h_1, \ldots, h_{\ell}$ simple co-roots ($h_i= \alpha_{i}^{\vee}$, for $i=1, \ldots , \ell$), and by $\omega _1 , \ldots, \omega_{\ell}$ fundametal weights for $\mathfrak{g}$ ($\ell = \dim \mathfrak{h}$). Let $\hat{\mathfrak{g}}$ be the (untwisted) affinization of $\mathfrak{g}$. Let  $\alpha_0, \alpha_1, \ldots, \alpha_{\ell}$ be simple roots, $\alpha_{0}^{\vee}, \alpha_{1}^{\vee} \ldots, \alpha_{\ell}^{\vee}$ simple co-roots, and $\Lambda _0 , \Lambda _1, \ldots, \Lambda_{\ell}$ fundametal weights for $\hat{\mathfrak{g}}$. For $\mu \in \mathfrak{h}^*$ and $k \in \mathbb{C}$, denote by $L_{k}\left( \mu \right)$ the irreducible highest weight $\hat{\mathfrak{g}}$--module with highest weight $\widehat{\mu} : = k \Lambda _0 + \mu  \in \hat{\mathfrak{h}} ^*$. 

Denote by $V^{k}(\mathfrak{g})$ the universal affine vertex algebra associated to simple Lie algebra $\mathfrak{g}$ and level $k \in \mathbb{C}$, $k \neq -h^{\vee}$, and 
by $L_{k}(\mathfrak{g})$ the unique simple quotient of $V^{k}(\mathfrak{g})$.
For any quotient $V$ of $V^{k}(\mathfrak{g})$, we define category $KL_{k}$ of $V$--modules as in \cite{AKMPP-IMRN}. If $\mathfrak{g}$ is a $1$--dimensional commutative Lie algebra and $k\ne 0$, then $V^{k}(\mathfrak{g})$ is a simple Heisenberg vertex algebra, which we denote by $M(1)$.

Let $v$ be a $\hat{\mathfrak{g}}$--singular vector in $V^{k}(\mathfrak{g})$. Denote by
$\langle v \rangle$ the ideal in $V^{k}(\mathfrak{g})$ generated by $v$ and by
\begin{equation} \label{L-tilde}
	\widetilde{L}_{k}\left( \mathfrak{g}\right)=V^{k}(\mathfrak{g})/ \langle v \rangle
\end{equation}
the associated quotient vertex algebra.

\subsection{Zhu's algebra}
For a vertex algebra $V$, denote by $A(V)$ Zhu's algebra associated to $V$ (cf. \cite{Z}). Denote by $[a]$ the image of $a \in V$ in $A(V)$. We have:

\begin{proposition}\cite[Theorem 2.1.2, Theorem 2.2.1]{Z}  \label{Zhu1}
	\item[(1)] Let $M=\oplus_{n=0}^{\infty}M(n)$ be a $\mathbb{Z}_+$--graded $V$--module. Then $M(0)$ is an $A(V)$--module.
	\item[(2)] Let $W$ be an $A(V)$--module. Then there exists a $\mathbb{Z}_+$--graded  $V$-module $M=\oplus_{n=0}^{\infty}M(n)$ such that $M(0)$ is isomorphic to $W$ as an $A(V)$--module.
\end{proposition}

\begin{proposition}\cite[Theorem 2.2.2]{Z}
	There is a one-to-one correspondence between irreducible $A(V)$--modules and irreducible $\mathbb{Z}_+$--graded $V$--modules.
\end{proposition}

It follows from \cite[Theorem 3.1.1]{FZ} that the Zhu's algebra  $A(V^{k}(\mathfrak{g}))$ is isomorphic to universal enveloping algebra $\mathcal{U}(\mathfrak{g})$, where the isomorphism $F \colon A(V^{k}(\mathfrak{g})) \to \mathcal{U}(\mathfrak{g})$ is given by
\begin{equation} \label{Zhu-iso} F ([a_1(-n_1-1) \ldots a_m(-n_m-1)\mathbf{1}]) = (-1)^{n_1+ \cdots + n_m}a_m \ldots a_1,
\end{equation}
for $a_1, \ldots ,a_m \in \mathfrak{g}$ and $n_1, \ldots , n_m \in \mathbb{Z}_{\ge 0}$.
For the quotient vertex algebra $\widetilde{L}_{-5/2}( \mathfrak{g})$ defined by relation (\ref{L-tilde}), the associated Zhu's algebra is 
$$ A(\widetilde{L}_{k}( \mathfrak{g})) \cong \mathcal{U}(\mathfrak{g})/ \langle v' \rangle,$$
where $\langle v' \rangle$ is a two-sided ideal in $\mathcal{U}(\mathfrak{g})$ generated by the vector $v'= F([v])$ (cf. \cite[Proposition 1.4.2]{FZ}). 

We recall the method for classification of irreducible $A(\widetilde{L}_{k}( \mathfrak{g}))$--modules in the category $\mathcal{O}$ from \cite{A-PhD, AM}.
Denote by $_L$ the adjoint action of  $\mathcal{U}(\mathfrak g)$ on
$\mathcal{U}(\mathfrak g)$ defined by $ X_Lf=[X,f]$ for $X \in
\mathfrak g$ and $f \in \mathcal{U}(\mathfrak g)$. Let $R$ be a $\mathcal{U}(\mathfrak g)$--submodule
of $\mathcal{U}(\mathfrak g)$ generated by the vector $v'$.
Clearly, $R$ is an irreducible finite-dimensional $\mathcal{U}(\mathfrak g)$--module. 
Let $R_{0}$ be the zero-weight subspace of $R$.

\begin{proposition}\cite{A-PhD, AM} \label{Class-affine-quotient}
	Let $V(\mu)$ be an irreducible highest weight $\mathcal{U}(\mathfrak{g})$--module with the highest weight vector  $v_{\mu}$, for $\mu \in \mathfrak{h}^*$. The following statements are equivalent:
	\begin{enumerate}
		\item $V(\mu)$ is an $A( \widetilde{L}_{k}( \mathfrak{g}) ) $--module,
		\item $RV(\mu )=0$,
		\item $R_0v_{\mu}=0$.
	\end{enumerate}
\end{proposition}

Denote by $\mathcal{U}(\mathfrak{g} ) _0$ the zero-weight subspace of $\mathcal{U}(\mathfrak{g} )$, and consider the Harish-Chandra projection map 
$$ \mathcal{U}(\mathfrak{g} ) _0 \to \mathcal{U}(\mathfrak{h} )$$
which is the restriction of the projection map
$$\mathcal{U}(\mathfrak{g} ) = \mathcal{U}(\mathfrak{h} ) \oplus (\mathfrak{n}_{-}\mathcal{U}(\mathfrak{g}) + \mathcal{U}(\mathfrak{g}) \mathfrak{n}_{+}) \to \mathcal{U}(\mathfrak{h} )$$
to $\mathcal{U}(\mathfrak{g} ) _0$ (cf. \cite{ArMo-sheets}). 

Let $r \in R_{0}$. Since $R_{0} \subset \mathcal{U}(\mathfrak{g} ) _0$, it follows that there exists the unique polynomial
$p_{r} \in \mathcal{S}( \frak h) = \mathcal{U}( \frak h)$ such that
$$ rv_{\mu}=p_{r}(\mu)v_{\mu}.$$
Set 
\begin{equation} \label{def-polinomi}
 {\mathcal P}_{0}=\{ \ p_{r} \ \vert \ r \in R_{0} \}. 
\end{equation}
We have:

\begin{corollary}\cite{A-PhD, AM} \label{koro-polinomi} There is a one-to-one correspondence between 
	\begin{enumerate} 
		\item irreducible $A(\widetilde{L}_{k}( \mathfrak{g}))$--modules
		in the category $\mathcal{O}$, 
		\item weights $\mu \in {\frak h}^{*}$ such that
		$p(\mu)=0$ for all $p \in {\mathcal P}_{0}$.
	\end{enumerate}
\end{corollary}

\subsection{Zhu's $C_2$--algebra} \label{subsect-C2-affine}
For a vertex algebra $V$, denote by $R_V = V/C_2(V)$ Zhu's $C_2$--algebra of $V$ (cf. \cite{Z}). Then 
$$ R_{V^{k}(\mathfrak{g})} \cong \mathcal{S}(\mathfrak{g}) $$
under the algebra isomorphism uniquely determined by 
\begin{equation} \label{C2-map}
	\overline{x(-1)\mathbf{1}} \mapsto x, \quad \mbox{for} \ x \in \mathfrak{g}, 
\end{equation}
where $\overline{w}$ denotes the image of $w$ in $R_{V^{k}(\mathfrak{g})}$, for any $w \in V^{k}(\mathfrak{g})$. For the quotient vertex algebra $\widetilde{L}_{k}( \mathfrak{g})$ defined by relation (\ref{L-tilde}), denote by $v''$ the image of vector $\overline{v}$ under the map (\ref{C2-map}). Then 
$$ R_{\widetilde{L}_{k}( \mathfrak{g})} \cong \mathcal{S}(\mathfrak{g}) / I_{W},  $$
where $W$ is a $\mathfrak{g}$--module generated by $v''$ under the adjoint action, and $I_W$ is the ideal of $\mathcal{S}(\mathfrak{g})$ generated by $W$ (cf. \cite{ArMo-joseph}, \cite{ArMo-sheets}).

\subsection{Affine $\mathcal{W}$--algebras} \label{subsect-W-alg} In this subsection we briefly recall certain results on affine $\mathcal{W}$--algebras (see \cite{KW3} for details).

Let $\mathfrak{g}$ be a simple Lie algebra, and $(x,f)$ a pair of elements of 
$\mathfrak{g}$ such that $f$ is nilpotent, $\mathrm{ad} \, x$ acts semisimply on 
$\mathfrak{g}$ with half-integer eigenvalues: 
\begin{equation} \label{rel-grading}
\mathfrak{g} = \bigoplus _{j \in \frac{1}{2} \mathbb{Z}} \mathfrak{g} _j,	
\end{equation} 
and that this grading is good with respect to $f$. Let $W^k(\g, x, f)$ be the affine $\mathcal{W}$--algebra associated with $\mathfrak{g}$, $x$ and $f$ at level $k$, obtained by generalized Drinfeld-Sokolov reduction. We will assume that $k \neq -h^{\vee}$. Let us denote $W^k(\g, x, f)$ by $W^k(\g, f)$ for short. Thus
$$ W^k(\g, f)= H_f(V^{k}(\mathfrak{g})).$$
Let us also denote by $W_k(\g, f)$ the unique simple quotient of $W^k(\g, f)$.

Let $\{u_1, \ldots , u_d\}$ be a basis of the centralizer 
$\mathfrak{g}^f$, such that $u_i \in \mathfrak{g} _{-j_i}$, for all $i=1, \ldots ,d$,
where $j_i$ are some non-negative half-integers. Then the $\mathcal{W}$--algebra $W^k(\g, f)$ is strongly generated by elements
$$ J^{\{u_i\}} = J^{(u_i)} + \mathrm{(lower \ terms)}, $$
where
$$ J^{(u_i)} = u_i + \mathrm{(charged \ fermionic  \ part)},  $$
such that the conformal weight of $J^{\{u_i\}}$ is equal to $j_i +1$, for $i=1, \ldots ,d$ (see \cite{KW3} for details).

Next, we recall the description of Zhu's $C_2$--algebra $R_{W^k(\g, f)}$ from \cite{DSK}, \cite{Ara-annals} (see also \cite{ArMo-sheets}). Let $ \chi=\left( f \, | \, \cdot \right) \in \g ^*$. Choose a Lagrangian subspace $\mathcal{L} \subset \mathfrak{g} _{\frac{1}{2}}$ and set 
\begin{equation} \label{rel-ideal-J}
\mathfrak{m}=\mathcal{L} \oplus \bigoplus_{j\geq 1} \g _j, \ \ \ \  J_\chi = \sum_{x \in \mathfrak{m} } \mathcal{S}(\g)(x-\chi(x)).
\end{equation} 
Let $M$ be the unipotent subgroup of $G$ corresponding to $\mathfrak{m}$.
\begin{proposition}[\cite{DSK}, \cite{Ara-annals}] \label{prop-C2-W} We have:
$$ R_{W^k(\g, f)} \cong \left( \mathcal{S}(\g)/ J_\chi  \right) ^M. $$
\end{proposition}	

\subsection{Affine $\mathcal{W}$--algebra $W^k(sl(4), f_{subreg})$} \label{subsect-W-subreg}
In this subsection
we apply results stated in Subsection \ref{subsect-W-alg} to the
affine $\mathcal{W}$--algebra $W^k(sl(4), f_{subreg})$. This vertex algebra was recently studied in \cite{CL}, \cite{G-2017}.

Let $\mathfrak{g}$ be the simple Lie algebra $sl(4)$, let
\begin{equation} \label{rel-f-subreg}
	f=f_{subreg}=f_{\varepsilon_2-\varepsilon_3} + f_{\varepsilon_3-\varepsilon_4}
\end{equation}
be the subregular nilpotent element, and 
\begin{equation} \label{rel-x-grading}
	x= \frac{1}{4} \left(3 h_1 + 6  h_2 + 5 h_3\right) = \omega _2 + \omega _3
\end{equation}
a semisimple element of $\mathfrak{g}$ which defines a good grading with respect to $f$ (cf. \cite{G-2017}). The associated grading (\ref{rel-grading}) is also even, i.e. $\mathfrak{g} _j =0$ for $j \notin \mathbb{Z}$. In this case $\dim \mathfrak{g}^f =5$, so $W^k(\mathfrak{g}, f_{subreg})$ is strongly generated by $5$ elements,
which we denote by $J, \bar{L}, W, G^+, G^-$, and their conformal weights are $1,2,3,1,3$, respectively.  The explicit OPE for generators of $W^k(\mathfrak{g}, f_{subreg})$ will be given in the Appendix.

 Note that, since  $e_{\varepsilon_1-\varepsilon_2} \in \mathfrak{g}^f _{0}$,  we have the associated generator 
$$G^+ = J^{\{e_{\varepsilon_1-\varepsilon_2}\}} = J^{(e_{\varepsilon_1-\varepsilon_2})}$$
of conformal weight $1$. More details on these generators will be given in Section \ref{sect-collapsing}.

\section{Affine vertex algebra associated to $\widehat{sl(4)}$ at level $-5/2$ }
\label{sect-singular-vector}

In this section we determine an explicit formula for a singular vector of conformal weight four in the affine vertex algebra $V^{-5/2}(sl(4))$. We explicitly determine Zhu's algebra and Zhu's $C_2$--algebra of the associated quotient vertex algebra. 
Throughout this section we denote by $\mathfrak{g}$ the simple Lie algebra $sl(4)$, and use the standard choice of root vectors for $sl(4)$.

\begin{theorem} \label{thm-singv}
	The following vector $v$ is a singular vector of weight $-\frac{5}{2}\Lambda _0 - 4 \delta +  2\omega_2$  in $V^{-5/2}( \mathfrak{g}) $:
	\begin{eqnarray*}
	&&v= e_{\varepsilon_1-\varepsilon_3}(-1)e_{\varepsilon_2-\varepsilon_4}(-3)\mathbf{1}
	+e_{\varepsilon_1-\varepsilon_3}(-3)e_{\varepsilon_2-\varepsilon_4}(-1)\mathbf{1}\\
	&&+\frac{1}{2}e_{\varepsilon_1-\varepsilon_3}(-2)e_{\varepsilon_2-\varepsilon_4}(-2)\mathbf{1}
	-e_{\varepsilon_1-\varepsilon_4}(-1)e_{\varepsilon_2-\varepsilon_3}(-3)\mathbf{1}\\
	&&-e_{\varepsilon_1-\varepsilon_4}(-3)e_{\varepsilon_2-\varepsilon_3}(-1)\mathbf{1}
	-\frac{1}{2}e_{\varepsilon_1-\varepsilon_4}(-2)e_{\varepsilon_2-\varepsilon_3}(-2)\mathbf{1}\\
	&&+e_{\varepsilon_2-\varepsilon_4}(-1)e_{\varepsilon_2-\varepsilon_3}(-2)e_{\varepsilon_1-\varepsilon_2}(-1)\mathbf{1}
	-e_{\varepsilon_2-\varepsilon_4}(-2)e_{\varepsilon_2-\varepsilon_3}(-1)e_{\varepsilon_1-\varepsilon_2}(-1)\mathbf{1}\\
	&&-e_{\varepsilon_1-\varepsilon_3}(-1)e_{\varepsilon_2-\varepsilon_3}(-2)e_{\varepsilon_3-\varepsilon_4}(-1)\mathbf{1}
	-3e_{\varepsilon_1-\varepsilon_3}(-2)e_{\varepsilon_2-\varepsilon_3}(-1)e_{\varepsilon_3-\varepsilon_4}(-1)\mathbf{1}\\
	&&+2e_{\varepsilon_1-\varepsilon_2}(-1)e_{\varepsilon_2-\varepsilon_3}(-1)^2e_{\varepsilon_3-\varepsilon_4}(-1)\mathbf{1}
	-\frac{2}{3}e_{\varepsilon_1-\varepsilon_3}(-1)e_{\varepsilon_2-\varepsilon_4}(-1)h_{2}(-2)\mathbf{1}\\
	&&-e_{\varepsilon_1-\varepsilon_3}(-1)e_{\varepsilon_2-\varepsilon_4}(-2)h_{1}(-1)\mathbf{1}
	-e_{\varepsilon_1-\varepsilon_3}(-1)e_{\varepsilon_2-\varepsilon_4}(-2)h_{2}(-1)\mathbf{1}\\
	&&-e_{\varepsilon_1-\varepsilon_3}(-2)e_{\varepsilon_2-\varepsilon_4}(-1)h_{2}(-1)\mathbf{1}
	-e_{\varepsilon_1-\varepsilon_3}(-2)e_{\varepsilon_2-\varepsilon_4}(-1)h_{3}(-1)\mathbf{1}\\
	&&+\frac{2}{3}e_{\varepsilon_1-\varepsilon_4}(-1)e_{\varepsilon_2-\varepsilon_3}(-1)h_{2}(-2)\mathbf{1}
	+e_{\varepsilon_1-\varepsilon_4}(-1)e_{\varepsilon_2-\varepsilon_3}(-2)h_{1}(-1)\mathbf{1}\\
	&&+e_{\varepsilon_1-\varepsilon_4}(-1)e_{\varepsilon_2-\varepsilon_3}(-2)h_{2}(-1)\mathbf{1}
	+e_{\varepsilon_1-\varepsilon_4}(-1)e_{\varepsilon_2-\varepsilon_3}(-2)h_{3}(-1)\mathbf{1}\\
	&&+e_{\varepsilon_1-\varepsilon_4}(-2)e_{\varepsilon_2-\varepsilon_3}(-1)h_{2}(-1)\mathbf{1}
	+\frac{2}{3}e_{\varepsilon_1-\varepsilon_3}(-1)e_{\varepsilon_2-\varepsilon_4}(-1)h_{1}(-1)h_{2}(-1)\mathbf{1}\\
	&&-\frac{2}{3}e_{\varepsilon_1-\varepsilon_3}(-1)e_{\varepsilon_2-\varepsilon_4}(-1)h_{1}(-1)h_{3}(-1)\mathbf{1}
	+\frac{2}{3}e_{\varepsilon_1-\varepsilon_3}(-1)e_{\varepsilon_2-\varepsilon_4}(-1)h_{2}(-1)^2\mathbf{1}\\
	&&+\frac{2}{3}e_{\varepsilon_1-\varepsilon_3}(-1)e_{\varepsilon_2-\varepsilon_4}(-1)h_{2}(-1)h_{3}(-1)\mathbf{1}
	-\frac{2}{3}e_{\varepsilon_1-\varepsilon_4}(-1)e_{\varepsilon_2-\varepsilon_3}(-1)h_{1}(-1)h_{2}(-1)\mathbf{1}\\
	&&-\frac{4}{3}e_{\varepsilon_1-\varepsilon_4}(-1)e_{\varepsilon_2-\varepsilon_3}(-1)h_{1}(-1)h_{3}(-1)\mathbf{1}
	-\frac{2}{3}e_{\varepsilon_1-\varepsilon_4}(-1)e_{\varepsilon_2-\varepsilon_3}(-1)h_{2}(-1)^2\mathbf{1}\\
	&&-\frac{2}{3}e_{\varepsilon_1-\varepsilon_4}(-1)e_{\varepsilon_2-\varepsilon_3}(-1)h_{2}(-1)h_{3}(-1)\mathbf{1}
	-\frac{2}{3}e_{\varepsilon_1-\varepsilon_3}(-1)e_{\varepsilon_2-\varepsilon_4}(-1)e_{\varepsilon_1-\varepsilon_2}(-1)f_{\varepsilon_1-\varepsilon_2}(-1)\mathbf{1}\\
	&&+\frac{4}{3}e_{\varepsilon_1-\varepsilon_3}(-1)e_{\varepsilon_2-\varepsilon_4}(-1)e_{\varepsilon_1-\varepsilon_3}(-1)f_{\varepsilon_1-\varepsilon_3}(-1)\mathbf{1}
	+\frac{4}{3}e_{\varepsilon_1-\varepsilon_3}(-1)e_{\varepsilon_2-\varepsilon_4}(-1)e_{\varepsilon_1-\varepsilon_4}(-1)f_{\varepsilon_1-\varepsilon_4}(-1)\mathbf{1}\\
	&&+\frac{4}{3}e_{\varepsilon_1-\varepsilon_3}(-1)e_{\varepsilon_2-\varepsilon_4}(-1)e_{\varepsilon_2-\varepsilon_3}(-1)f_{\varepsilon_2-\varepsilon_3}(-1)\mathbf{1}
	+\frac{4}{3}e_{\varepsilon_1-\varepsilon_3}(-1)e_{\varepsilon_2-\varepsilon_4}(-1)^2f_{\varepsilon_2-\varepsilon_4}(-1)\mathbf{1}\\
	&&-\frac{2}{3}e_{\varepsilon_1-\varepsilon_3}(-1)e_{\varepsilon_2-\varepsilon_4}(-1)e_{\varepsilon_3-\varepsilon_4}(-1)f_{\varepsilon_3-\varepsilon_4}(-1)\mathbf{1}
	+\frac{2}{3}e_{\varepsilon_1-\varepsilon_4}(-1)e_{\varepsilon_2-\varepsilon_3}(-1)e_{\varepsilon_1-\varepsilon_2}(-1)f_{\varepsilon_1-\varepsilon_2}(-1)\mathbf{1}\\
	&&-\frac{4}{3}e_{\varepsilon_1-\varepsilon_4}(-1)e_{\varepsilon_2-\varepsilon_3}(-1)e_{\varepsilon_1-\varepsilon_3}(-1)f_{\varepsilon_1-\varepsilon_3}(-1)\mathbf{1}
	-\frac{4}{3}e_{\varepsilon_1-\varepsilon_4}(-1)e_{\varepsilon_2-\varepsilon_3}(-1)e_{\varepsilon_1-\varepsilon_4}(-1)f_{\varepsilon_1-\varepsilon_4}(-1)\mathbf{1}\\
	&&-\frac{4}{3}e_{\varepsilon_1-\varepsilon_4}(-1)e_{\varepsilon_2-\varepsilon_3}(-1)^2f_{\varepsilon_2-\varepsilon_3}(-1)\mathbf{1}
	-\frac{4}{3}e_{\varepsilon_1-\varepsilon_4}(-1)e_{\varepsilon_2-\varepsilon_3}(-1)e_{\varepsilon_2-\varepsilon_4}(-1)f_{\varepsilon_2-\varepsilon_4}(-1)\mathbf{1}\\
	&&+\frac{2}{3}e_{\varepsilon_1-\varepsilon_4}(-1)e_{\varepsilon_2-\varepsilon_3}(-1)e_{\varepsilon_3-\varepsilon_4}(-1)f_{\varepsilon_3-\varepsilon_4}(-1)\mathbf{1}
	-2e_{\varepsilon_2-\varepsilon_4}(-1)e_{\varepsilon_2-\varepsilon_3}(-1)e_{\varepsilon_1-\varepsilon_2}(-1)h_{3}(-1)\mathbf{1}\\
	&&+2e_{\varepsilon_1-\varepsilon_3}(-1)e_{\varepsilon_2-\varepsilon_3}(-1)e_{\varepsilon_3-\varepsilon_4}(-1)h_{1}(-1)\mathbf{1}
	+2e_{\varepsilon_1-\varepsilon_3}(-1)e_{\varepsilon_1-\varepsilon_4}(-1)f_{\varepsilon_1-\varepsilon_2}(-1)h_{3}(-1)\mathbf{1}\\
	&&+e_{\varepsilon_1-\varepsilon_3}(-1)e_{\varepsilon_1-\varepsilon_4}(-2)f_{\varepsilon_1-\varepsilon_2}(-1)\mathbf{1}
	-e_{\varepsilon_1-\varepsilon_3}(-2)e_{\varepsilon_1-\varepsilon_4}(-1)f_{\varepsilon_1-\varepsilon_2}(-1)\mathbf{1}\\
	&&-2e_{\varepsilon_1-\varepsilon_3}(-1)^2e_{\varepsilon_3-\varepsilon_4}(-1)f_{\varepsilon_1-\varepsilon_2}(-1)\mathbf{1}
	-2e_{\varepsilon_1-\varepsilon_4}(-1)e_{\varepsilon_2-\varepsilon_4}(-1)f_{\varepsilon_3-\varepsilon_4}(-1)h_{1}(-1)\mathbf{1}\\
	&&+e_{\varepsilon_1-\varepsilon_4}(-1)e_{\varepsilon_2-\varepsilon_4}(-2)f_{\varepsilon_3-\varepsilon_4}(-1)\mathbf{1}
	-e_{\varepsilon_1-\varepsilon_4}(-2)e_{\varepsilon_2-\varepsilon_4}(-1)f_{\varepsilon_3-\varepsilon_4}(-1)\mathbf{1}\\
	&&-2e_{\varepsilon_2-\varepsilon_4}(-1)^2e_{\varepsilon_1-\varepsilon_2}(-1)f_{\varepsilon_3-\varepsilon_4}(-1)\mathbf{1}
	+2e_{\varepsilon_1-\varepsilon_4}(-1)^2f_{\varepsilon_1-\varepsilon_2}(-1)f_{\varepsilon_3-\varepsilon_4}(-1) \mathbf{1}.
\end{eqnarray*}
\end{theorem}

\begin{proof}
	Direct verification of relations $e_{\varepsilon_i-\varepsilon_{i+1}}(0).v=0$ for $i=1,2,3$ and $f_\theta(1).v=0$.
\end{proof}

\begin{remark} The expression for singular vector  in Theorem \ref{thm-singv} was also known by P. Moseneder Frajria and P. Papi.
\end{remark}
\begin{remark}
As a $\hat{\mathfrak{g}}$--module, $V^{-5/2}( \mathfrak{g})$ is clearly a quotient of the Verma module $M(-\frac{5}{2} \Lambda_0)$. Thus, when looking for singular vectors in $V^{-5/2}( \mathfrak{g})$, it is natural to consider the projection of singular vectors from $M(-\frac{5}{2} \Lambda_0)$ onto $V^{-5/2}( \mathfrak{g})$. Using the results from \cite{KK}, one can easily verify that there exists a singular vector of weight $-\frac{5}{2}\Lambda _0 - 4 \delta +  2\omega_2$ in $M(-\frac{5}{2} \Lambda_0)$. But using the formula for that singular vector from \cite{MFF}, one easily obtains that it projects to zero in 
$V^{-5/2}( \mathfrak{g})$. Thus, the vector $v$ from Theorem \ref{thm-singv} is a subsingular (or primitive) vector in $M(-\frac{5}{2} \Lambda_0)$ with the same weight $-\frac{5}{2}\Lambda _0 - 4 \delta +  2\omega_2$ as the singular 
vector in $M(-\frac{5}{2} \Lambda_0)$. 
\end{remark}

Let us denote by 
$$\widetilde{L}_{-5/2}( \mathfrak{g})=V^{-5/2}( \mathfrak{g})/ \langle v \rangle $$
the associated quotient vertex algebra.
\begin{remark} \label{automorphism}
	Vertex algebra $V^{-5/2}( \mathfrak{g}) $ has an order two automorphism $\sigma$ which is lifted from the automorphism of the Dynkin diagram of $\mathfrak{g}$, defined by:
	$$\sigma(\alpha_1) = \alpha_3, \ \sigma(\alpha_2) = \alpha_2, \ \sigma(\alpha_3) =  \alpha_1.$$ 
	One easily checks that $\sigma (v) =v$ for the singular vector $v$ from Theorem \ref{thm-singv}. This implies that $\sigma$ induces an automorphism of $\widetilde{L}_{-5/2}(\mathfrak{g})$.  This fact will be used in Section \ref{sect-collapsing}. 
\end{remark}

In the next propositon we determine Zhu's algebra of the vertex algebra $\widetilde{L}_{-5/2}( \mathfrak{g})$.

\begin{proposition} \label{Zhu-algebra}
Zhu's algebra $A( \widetilde{L}_{-5/2}( \mathfrak{g}))$ is isomorphic to $\mathcal{U}(\mathfrak{g})/\langle v' \rangle$, where $\langle v' \rangle$ is a two-sided ideal in $\mathcal{U}(\mathfrak{g})$ generated by the following vector $v'$:
\begin{eqnarray*}
	&&v'=\frac{5}{2}e_{\varepsilon_2-\varepsilon_4}e_{\varepsilon_1-\varepsilon_3} - \frac{5}{2}e_{\varepsilon_2-\varepsilon_3}e_{\varepsilon_1-\varepsilon_4} + 4e_{\varepsilon_3-\varepsilon_4}e_{\varepsilon_2-\varepsilon_3}e_{\varepsilon_1-\varepsilon_3} + 2e_{\varepsilon_3-\varepsilon_4}e_{\varepsilon_2-\varepsilon_3}^2e_{\varepsilon_1-\varepsilon_2}\\
	&&+\frac{8}{3}h_{2}e_{\varepsilon_2-\varepsilon_4}e_{\varepsilon_1-\varepsilon_3} + h_{1}e_{\varepsilon_2-\varepsilon_4}e_{\varepsilon_1-\varepsilon_3} + h_{3}e_{\varepsilon_2-\varepsilon_4}e_{\varepsilon_1-\varepsilon_3}\\
	&&-\frac{8}{3}h_{2}e_{\varepsilon_2-\varepsilon_3}e_{\varepsilon_1-\varepsilon_4} - h_{1}e_{\varepsilon_2-\varepsilon_3}e_{\varepsilon_1-\varepsilon_4} - h_{3}e_{\varepsilon_2-\varepsilon_3}e_{\varepsilon_1-\varepsilon_4}\\
	&&+\frac{2}{3}h_{1}h_{2}e_{\varepsilon_2-\varepsilon_4}e_{\varepsilon_1-\varepsilon_3} - \frac{2}{3}h_{1}h_{3}e_{\varepsilon_2-\varepsilon_4}e_{\varepsilon_1-\varepsilon_3}+\frac{2}{3}h_{2}^2e_{\varepsilon_2-\varepsilon_4}e_{\varepsilon_1-\varepsilon_3}
	+\frac{2}{3}h_{2}h_{3}e_{\varepsilon_2-\varepsilon_4}e_{\varepsilon_1-\varepsilon_3}\\
	&&-\frac{2}{3}h_{1}h_{2}e_{\varepsilon_2-\varepsilon_3}e_{\varepsilon_1-\varepsilon_4} - \frac{4}{3}h_{1}h_{3}e_{\varepsilon_2-\varepsilon_3}e_{\varepsilon_1-\varepsilon_4}-\frac{2}{3}h_{2}^2e_{\varepsilon_2-\varepsilon_3}e_{\varepsilon_1-\varepsilon_4} - \frac{2}{3}h_{2}h_{3}e_{\varepsilon_2-\varepsilon_3}e_{\varepsilon_1-\varepsilon_4}\\
	&&-\frac{2}{3}f_{\varepsilon_1-\varepsilon_2}e_{\varepsilon_1-\varepsilon_2}e_{\varepsilon_2-\varepsilon_4}e_{\varepsilon_1-\varepsilon_3} + \frac{4}{3}f_{\varepsilon_1-\varepsilon_3}e_{\varepsilon_1-\varepsilon_3}e_{\varepsilon_2-\varepsilon_4}e_{\varepsilon_1-\varepsilon_3} + 
	\frac{4}{3}f_{\varepsilon_1-\varepsilon_4}e_{\varepsilon_1-\varepsilon_4}e_{\varepsilon_2-\varepsilon_4}e_{\varepsilon_1-\varepsilon_3}\\
	&&+\frac{4}{3}f_{\varepsilon_2-\varepsilon_3}e_{\varepsilon_2-\varepsilon_3}e_{\varepsilon_2-\varepsilon_4}e_{\varepsilon_1-\varepsilon_3} +
	\frac{4}{3}f_{\varepsilon_2-\varepsilon_4}e_{\varepsilon_2-\varepsilon_4}^2e_{\varepsilon_1-\varepsilon_3} - 
	\frac{2}{3}f_{\varepsilon_3-\varepsilon_4}e_{\varepsilon_3-\varepsilon_4}e_{\varepsilon_2-\varepsilon_4}e_{\varepsilon_1-\varepsilon_3}\\
	&&+\frac{2}{3}f_{\varepsilon_1-\varepsilon_2}e_{\varepsilon_1-\varepsilon_2}e_{\varepsilon_2-\varepsilon_3}e_{\varepsilon_1-\varepsilon_4} -
	\frac{4}{3}f_{\varepsilon_1-\varepsilon_3}e_{\varepsilon_1-\varepsilon_3}e_{\varepsilon_2-\varepsilon_3}e_{\varepsilon_1-\varepsilon_4} -
	\frac{4}{3}f_{\varepsilon_1-\varepsilon_4}e_{\varepsilon_1-\varepsilon_4}e_{\varepsilon_2-\varepsilon_3}e_{\varepsilon_1-\varepsilon_4}\\
	&&-\frac{4}{3}f_{\varepsilon_2-\varepsilon_3}e_{\varepsilon_2-\varepsilon_3}^2e_{\varepsilon_1-\varepsilon_4} -
	\frac{4}{3}f_{\varepsilon_2-\varepsilon_4}e_{\varepsilon_2-\varepsilon_4}e_{\varepsilon_2-\varepsilon_3}e_{\varepsilon_1-\varepsilon_4} +
	\frac{2}{3}f_{\varepsilon_3-\varepsilon_4}e_{\varepsilon_3-\varepsilon_4}e_{\varepsilon_2-\varepsilon_3}e_{\varepsilon_1-\varepsilon_4}\\
	&&-2h_{3}e_{\varepsilon_1-\varepsilon_2}e_{\varepsilon_2-\varepsilon_3}e_{\varepsilon_2-\varepsilon_4} + 
	2h_{1}e_{\varepsilon_3-\varepsilon_4}e_{\varepsilon_2-\varepsilon_3}e_{\varepsilon_1-\varepsilon_3} + 
	2h_{3}f_{\varepsilon_1-\varepsilon_2}e_{\varepsilon_1-\varepsilon_4}e_{\varepsilon_1-\varepsilon_3}\\
	&&-2f_{\varepsilon_1-\varepsilon_2}e_{\varepsilon_3-\varepsilon_4}e_{\varepsilon_1-\varepsilon_3}^2 -
	2h_{1}f_{\varepsilon_3-\varepsilon_4}e_{\varepsilon_2-\varepsilon_4}e_{\varepsilon_1-\varepsilon_4} -
	2f_{\varepsilon_3-\varepsilon_4}e_{\varepsilon_1-\varepsilon_2}e_{\varepsilon_2-\varepsilon_4}^2 + 
	2f_{\varepsilon_3-\varepsilon_4}f_{\varepsilon_1-\varepsilon_2}e_{\varepsilon_1-\varepsilon_4}^2.
\end{eqnarray*}

\end{proposition}  
\begin{proof}
	Using relation (\ref{Zhu-iso}), one easily obtains that $v'= F([v])$. The claim now follows from \cite[Proposition 1.4.2, Theorem 3.1.1]{FZ}.
\end{proof}

The description of Zhu's $C_2$--algebra of vertex algebra $\widetilde{L}_{-5/2}( \mathfrak{g})$ follows from the results stated in Subsection \ref{subsect-C2-affine}.
\begin{proposition} \label{C2-algebra}
Zhu's $C_2$--algebra $R_{\widetilde{L}_{-5/2}( \mathfrak{g})}$ is isomorphic to $\mathcal{S}(\mathfrak{g}) / I_{W}$, where $I_W$ is the ideal of $\mathcal{S}(\mathfrak{g})$ generated by $W$, and $W$ is a $\mathfrak{g}$--module generated under the adjoint action by the following vector $v''$:
\begin{eqnarray*}
	&&v''=2e_{\varepsilon_3-\varepsilon_4}e_{\varepsilon_2-\varepsilon_3}^2e_{\varepsilon_1-\varepsilon_2} + \frac{2}{3}(h_1h_2-h_1h_3+h_2^2+h_2h_3)e_{\varepsilon_2-\varepsilon_4}e_{\varepsilon_1-\varepsilon_3}\\
	&&+\frac{2}{3}(-h_1h_2-2h_1h_3-h_2^2-h_2h_3)e_{\varepsilon_2-\varepsilon_3}e_{\varepsilon_1-\varepsilon_4}\\
	&&+\frac{2}{3}\Big(-f_{\varepsilon_1-\varepsilon_2}e_{\varepsilon_1-\varepsilon_2} +2f_{\varepsilon_1-\varepsilon_3}e_{\varepsilon_1-\varepsilon_3}   +2f_{\varepsilon_1-\varepsilon_4}e_{\varepsilon_1-\varepsilon_4} +2f_{\varepsilon_2-\varepsilon_3}e_{\varepsilon_2-\varepsilon_3}  \\
	&& \quad \quad +2f_{\varepsilon_2-\varepsilon_4}e_{\varepsilon_2-\varepsilon_4} 
	-f_{\varepsilon_3-\varepsilon_4}e_{\varepsilon_3-\varepsilon_4} \Big)     e_{\varepsilon_2-\varepsilon_4}e_{\varepsilon_1-\varepsilon_3}\\
	&&+\frac{2}{3}\Big(f_{\varepsilon_1-\varepsilon_2}e_{\varepsilon_1-\varepsilon_2} -2f_{\varepsilon_1-\varepsilon_3}e_{\varepsilon_1-\varepsilon_3}   -2f_{\varepsilon_1-\varepsilon_4}e_{\varepsilon_1-\varepsilon_4} -2f_{\varepsilon_2-\varepsilon_3}e_{\varepsilon_2-\varepsilon_3} \\
	&& \quad \quad -2f_{\varepsilon_2-\varepsilon_4}e_{\varepsilon_2-\varepsilon_4} 
	+f_{\varepsilon_3-\varepsilon_4}e_{\varepsilon_3-\varepsilon_4} \Big)     e_{\varepsilon_2-\varepsilon_3}e_{\varepsilon_1-\varepsilon_4}\\
	&&-2h_3e_{\varepsilon_1-\varepsilon_2}e_{\varepsilon_2-\varepsilon_3}e_{\varepsilon_2-\varepsilon_4} + 2h_1e_{\varepsilon_3-\varepsilon_4}e_{\varepsilon_2-\varepsilon_3}e_{\varepsilon_1-\varepsilon_3} + 2h_3f_{\varepsilon_1-\varepsilon_2}e_{\varepsilon_1-\varepsilon_4}e_{\varepsilon_1-\varepsilon_3}\\
	&&-2f_{\varepsilon_1-\varepsilon_2}e_{\varepsilon_3-\varepsilon_4}e_{\varepsilon_1-\varepsilon_3}^2  -2h_1f_{\varepsilon_3-\varepsilon_4}e_{\varepsilon_2-\varepsilon_4}e_{\varepsilon_1-\varepsilon_4} -2f_{\varepsilon_3-\varepsilon_4}e_{\varepsilon_1-\varepsilon_2}e_{\varepsilon_2-\varepsilon_4}^2 + 2f_{\varepsilon_3-\varepsilon_4}f_{\varepsilon_1-\varepsilon_2}e_{\varepsilon_1-\varepsilon_4}^2.
	\end{eqnarray*}
\end{proposition}

\begin{remark} 
	In what follows, we will prove that the vertex algebra $\widetilde{L}_{-5/2}( \mathfrak{g})$ is simple (see Section \ref{sect-collapsing}), so Proposition \ref{Zhu-algebra} and Proposition \ref{C2-algebra} will give the descriptions of Zhu's algebra $A(L_{-5/2}( \mathfrak{g}))$ and Zhu's $C_2$--algebra $R_{L_{-5/2}( \mathfrak{g})}$ of the simple vertex algebra $L_{-5/2}( \mathfrak{g})$.
\end{remark}

The following lemma will be useful in Section \ref{sect-collapsing}.
\begin{lemma} \label{lem-singv-modJ}
 Let $f_{subreg}$ be the subregular nilpotent element defined by relation 
(\ref{rel-f-subreg}), $x$ a semisimple element defined by (\ref{rel-x-grading}), 
and $J_\chi$ the associated ideal in $\mathcal{S}(\mathfrak{g})$
defined by relation (\ref{rel-ideal-J}). Then
$$ v'' \equiv 2e_{\varepsilon_1-\varepsilon_2} (\mathrm{mod} \ J_\chi). $$
\end{lemma}
\begin{proof}
Recall that the grading on $\mathfrak{g}$ given by the eigenspaces of $\mathrm{ad} \, x$ is even. In particular, $\mathfrak{g} _{\frac{1}{2}}=0$. Furthermore, 
\begin{eqnarray*}
	&& \g _1 =\mathrm{span}_{\mathbb{C}} \{ e_{\varepsilon_1-\varepsilon_3}, e_{\varepsilon_2-\varepsilon_3}, e_{\varepsilon_3-\varepsilon_4} \},\\
	&& \g _2=\mathrm{span}_{\mathbb{C}} \{e_{\varepsilon_1-\varepsilon_4}, e_{\varepsilon_2-\varepsilon_4}\},
\end{eqnarray*}
and $\mathfrak{g} _j =0$ for $j > 2$. Since $\chi=\left( f_{\varepsilon_2-\varepsilon_3}+f_{\varepsilon_3-\varepsilon_4} \, | \, \cdot \right)$, we obtain 
\begin{eqnarray} \label{rel-chi}
	&& \chi(e_{\varepsilon_2-\varepsilon_3})=\chi(e_{\varepsilon_3-\varepsilon_4})=1, \nonumber \\
	&&\chi(e_{\varepsilon_1-\varepsilon_3})=\chi(e_{\varepsilon_1-\varepsilon_4})=\chi(e_{\varepsilon_2-\varepsilon_4})=0.
\end{eqnarray}
Now, relations (\ref{rel-chi}) and the formula for $v''$ from Proposition \ref{C2-algebra} easily imply that 
$$ v'' \equiv 2e_{\varepsilon_1-\varepsilon_2} (\mathrm{mod} \ J_\chi). $$
\end{proof}

\section{Classification of irreducible $\widetilde{L}_{-5/2}(sl(4))$--modules in the category $\mathcal{O}$} \label{sect-class-quotient}

Let $\mathfrak{g}=sl(4)$. In this section we use Zhu's theory and the classification method stated in Proposition \ref{Class-affine-quotient} and Corollary \ref{koro-polinomi} to study the representation theory of the vertex algebra $\widetilde{L}_{-5/2}(\mathfrak{g})$, defined in Section~\ref{sect-singular-vector}. In particular, we obtain the classification of irreducible modules in the category ${\mathcal O}$ for that vertex algebra.

In the next lemma we determine a basis of the space of polynomials ${\mathcal P}_{0}$
defined by relation (\ref{def-polinomi}).

\begin{lemma} \label{lema-polinomi} We have
	$$ {\mathcal P}_{0} = \mbox{span}_{\mathbb{C}} \{p_1, p_2  \},$$
	where \begin{equation*} 
		\begin{aligned}
			&p_1(h)=-\frac{5}{2}h_2 - \frac{7}{2}h_1h_2 + \frac{3}{2}h_1h_3 - \frac{7}{2}h_2h_3 - \frac{31}{6}{h_2}^2 - \frac{13}{3}h_1{h_2}^2 - {h_1}^2h_2 - 2h_1h_2h_3 -\\
			&- \frac{10}{3}{h_2}^3 + h_1{h_3}^2 + {h_1}^2h_3 - \frac{13}{3}{h_2}^2h_3 - h_2{h_3}^2 - \frac{4}{3}h_1{h_2}^3 - \frac{2}{3}{h_1}^2{h_2}^2 +\\
			&+ \frac{2}{3}{h_1}^2{h_3}^2 - \frac{4}{3}{h_1}{h_2}^2h_3 - \frac{2}{3}{h_2}^4 - \frac{4}{3}{h_2}^3h_3 - \frac{2}{3}{h_2}^2{h_3}^2,
		\end{aligned}
	\end{equation*}
	and
	\begin{equation*} 
		\begin{aligned}
			&p_2(h)=\frac{5}{2}h_2 + \frac{7}{2}h_1h_2 + \frac{7}{2}h_2h_3 + \frac{31}{6}{h_2}^2 + \frac{13}{3}h_1{h_2}^2 + {h_1}^2h_2 + \frac{16}{3}h_1h_2h_3 +\\
			&+ \frac{10}{3}{h_2}^3 + \frac{13}{3}{h_2}^2h_3 + h_2{h_3}^2 + \frac{4}{3}h_1{h_2}^3 + \frac{2}{3}{h_1}^2{h_2}^2 + \frac{4}{3}{h_1}^2h_2h_3 +\\
			&+ \frac{8}{3}{h_1}{h_2}^2h_3 + \frac{4}{3}h_1h_2{h_3}^2 + \frac{2}{3}{h_2}^4 + \frac{4}{3}{h_2}^3h_3 + \frac{2}{3}{h_2}^2{h_3}^2.
		\end{aligned}
	\end{equation*}

\end{lemma}
\begin{proof}
	The $\mathcal{U}(\mathfrak{g})$--submodule $R$ of $\mathcal{U}(\mathfrak{g})$ generated by vector $v'$ (from Proposition \ref{Zhu-algebra}) under the adjoint action is clearly isomorphic to $V(2\omega_2)$. One easily obtains that dim$R_0=2$, which implies that dim${\mathcal P}_{0} \leq 2$. Let us denote
	\begin{align*}
		v_1=& \left( f_{\varepsilon_2-\varepsilon_4}f_{\varepsilon_1-\varepsilon_3} \right)_Lv', \\
		v_2=& \left( f_{\varepsilon_2-\varepsilon_3}f_{\varepsilon_1-\varepsilon_4}\right)_Lv'. 
	\end{align*}
	Then, the vectors $v_1$ and $v_2$ linearly span $R_0$. 
	
	Let us prove that $p_1 \in {\mathcal P}_{0}$. By direct calculation we determine the action of $\left(f_{\varepsilon_2-\varepsilon_4}f_{\varepsilon_1-\varepsilon_3} \right)_L$ on each monomial appearing in the formula for $v'$ from Proposition \ref{Zhu-algebra}:
	\begin{eqnarray*}
		&&\left(f_{\varepsilon_2-\varepsilon_4}f_{\varepsilon_1-\varepsilon_3}\right)_L e_{\varepsilon_2-\varepsilon_4}e_{\varepsilon_1-\varepsilon_3} \in  - h_1 h_2 - h_1 h_3 -h_2^2 - h_2 h_3 + \mathcal{U}(\mathfrak{g}) \mathfrak{n}_{+}\\
		&&
		\left(f_{\varepsilon_2-\varepsilon_4}f_{\varepsilon_1-\varepsilon_3}\right)_L h_1 e_{\varepsilon_2-\varepsilon_4}e_{\varepsilon_1-\varepsilon_3} \in  - h_1^2 h_2 - h_1^2 h_3 - h_1 h_2^2 - h_1 h_2 h_3 + \mathcal{U}(\mathfrak{g}) \mathfrak{n}_{+}\\
		&&
		\left(f_{\varepsilon_2-\varepsilon_4}f_{\varepsilon_1-\varepsilon_3}\right)_L h_2 e_{\varepsilon_2-\varepsilon_4}e_{\varepsilon_1-\varepsilon_3} \in  - h_1 h_2^2 - h_1 h_2 h_3 -h_2^3 - h_2^2 h_3 + \mathcal{U}(\mathfrak{g}) \mathfrak{n}_{+}\\
		&&
		\left(f_{\varepsilon_2-\varepsilon_4}f_{\varepsilon_1-\varepsilon_3}\right)_L
		h_3 e_{\varepsilon_2-\varepsilon_4}e_{\varepsilon_1-\varepsilon_3} \in  - h_1 h_2 h_3 - h_1 h_3^2 -h_2^2 h_3 - h_2 h_3^2 + \mathcal{U}(\mathfrak{g}) \mathfrak{n}_{+}\\
		&&
		\left(f_{\varepsilon_2-\varepsilon_4}f_{\varepsilon_1-\varepsilon_3}\right)_L h_1 h_2 e_{\varepsilon_2-\varepsilon_4}e_{\varepsilon_1-\varepsilon_3} \in  - h_1^2 h_2^2 - h_1^2 h_2 h_3 - h_1 h_2^3 - h_1 h_2^2 h_3 + \mathcal{U}(\mathfrak{g}) \mathfrak{n}_{+} \\
		&&
		\left(f_{\varepsilon_2-\varepsilon_4}f_{\varepsilon_1-\varepsilon_3}\right)_L h_1 h_3 e_{\varepsilon_2-\varepsilon_4}e_{\varepsilon_1-\varepsilon_3} \in  - h_1^2 h_2 h_3 - h_1^2 h_3^2 - h_1 h_2^2 h_3 - h_1 h_2 h_3^2 + \mathcal{U}(\mathfrak{g}) \mathfrak{n}_{+}\\
		&&
		\left(f_{\varepsilon_2-\varepsilon_4}f_{\varepsilon_1-\varepsilon_3}\right)_L h_2 ^2 e_{\varepsilon_2-\varepsilon_4}e_{\varepsilon_1-\varepsilon_3} \in  - h_1 h_2^3 - h_1 h_2^2 h_3 -h_2^4 - h_2^3 h_3 + \mathcal{U}(\mathfrak{g}) \mathfrak{n}_{+}\\
		&&
		\left(f_{\varepsilon_2-\varepsilon_4}f_{\varepsilon_1-\varepsilon_3}\right)_L h_2 h_3 e_{\varepsilon_2-\varepsilon_4}e_{\varepsilon_1-\varepsilon_3} \in  - h_1 h_2^2 h_3 - h_1 h_2 h_3^2 -h_2^3 h_3 - h_2^2 h_3^2 + \mathcal{U}(\mathfrak{g}) \mathfrak{n}_{+}\\
		&&
		\left(f_{\varepsilon_2-\varepsilon_4}f_{\varepsilon_1-\varepsilon_3}\right)_L 
		e_{\varepsilon_2-\varepsilon_3}e_{\varepsilon_1-\varepsilon_4} \in  h_2 + \mathcal{U}(\mathfrak{g}) \mathfrak{n}_{+}\\
		&&
		\left(f_{\varepsilon_2-\varepsilon_4}f_{\varepsilon_1-\varepsilon_3}\right)_L 
		h_1 e_{\varepsilon_2-\varepsilon_3}e_{\varepsilon_1-\varepsilon_4} \in h_1 h_2 + \mathcal{U}(\mathfrak{g}) \mathfrak{n}_{+} \\
		&&
		\left(f_{\varepsilon_2-\varepsilon_4}f_{\varepsilon_1-\varepsilon_3}\right)_L 
		h_2 e_{\varepsilon_2-\varepsilon_3}e_{\varepsilon_1-\varepsilon_4} \in h_2^2 + \mathcal{U}(\mathfrak{g}) \mathfrak{n}_{+} \\
		&&
		\left(f_{\varepsilon_2-\varepsilon_4}f_{\varepsilon_1-\varepsilon_3}\right)_L 
		h_3 e_{\varepsilon_2-\varepsilon_3}e_{\varepsilon_1-\varepsilon_4} \in h_2 h_3 + \mathcal{U}(\mathfrak{g}) \mathfrak{n}_{+} \\
		&&
		\left(f_{\varepsilon_2-\varepsilon_4}f_{\varepsilon_1-\varepsilon_3}\right)_L 
		h_1 h_2 e_{\varepsilon_2-\varepsilon_3}e_{\varepsilon_1-\varepsilon_4} \in h_1 h_2^2 + \mathcal{U}(\mathfrak{g}) \mathfrak{n}_{+} \\
		&&
		\left(f_{\varepsilon_2-\varepsilon_4}f_{\varepsilon_1-\varepsilon_3}\right)_L 
		h_1 h_3 e_{\varepsilon_2-\varepsilon_3}e_{\varepsilon_1-\varepsilon_4} \in h_1 h_2 h_3 + \mathcal{U}(\mathfrak{g}) \mathfrak{n}_{+}\\
		&&
		\left(f_{\varepsilon_2-\varepsilon_4}f_{\varepsilon_1-\varepsilon_3}\right)_L 
		h_2 ^2 e_{\varepsilon_2-\varepsilon_3}e_{\varepsilon_1-\varepsilon_4} \in h_2^3 + \mathcal{U}(\mathfrak{g}) \mathfrak{n}_{+} \\
		&&
		\left(f_{\varepsilon_2-\varepsilon_4}f_{\varepsilon_1-\varepsilon_3}\right)_L 
		h_2 h_3 e_{\varepsilon_2-\varepsilon_3}e_{\varepsilon_1-\varepsilon_4} \in h_2^2 h_3 + \mathcal{U}(\mathfrak{g}) \mathfrak{n}_{+} \\
		&&
		\left(f_{\varepsilon_2-\varepsilon_4}f_{\varepsilon_1-\varepsilon_3}\right)_L 
		e_{\varepsilon_3-\varepsilon_4} e_{\varepsilon_2-\varepsilon_3}e_{\varepsilon_1-\varepsilon_3} \in  h_1 h_3 + h_2 h_3 + \mathcal{U}(\mathfrak{g}) \mathfrak{n}_{+}\\
		&&
		\left(f_{\varepsilon_2-\varepsilon_4}f_{\varepsilon_1-\varepsilon_3}\right)_L 
		h_1 e_{\varepsilon_3-\varepsilon_4} e_{\varepsilon_2-\varepsilon_3}e_{\varepsilon_1-\varepsilon_3} \in h_1^2 h_3 + h_1 h_2 h_3 + \mathcal{U}(\mathfrak{g}) \mathfrak{n}_{+}\\
		&&
		\left(f_{\varepsilon_2-\varepsilon_4}f_{\varepsilon_1-\varepsilon_3}\right)_L 
		e_{\varepsilon_3-\varepsilon_4}  e_{\varepsilon_2-\varepsilon_3}^2 e_{\varepsilon_1-\varepsilon_2} \in -2 h_2 h_3 + \mathcal{U}(\mathfrak{g}) \mathfrak{n}_{+}\\
		&&
		\left(f_{\varepsilon_2-\varepsilon_4}f_{\varepsilon_1-\varepsilon_3}\right)_L 
		h_3 e_{\varepsilon_1-\varepsilon_2} e_{\varepsilon_2-\varepsilon_3}e_{\varepsilon_2-\varepsilon_4} \in - h_1 h_2 h_3 - h_1 h_3^2 + \mathcal{U}(\mathfrak{g}) \mathfrak{n}_{+}\\
		&&
		\left(f_{\varepsilon_2-\varepsilon_4}f_{\varepsilon_1-\varepsilon_3}\right)_L 
		f_{\varepsilon_1-\varepsilon_2}  e_{\varepsilon_1-\varepsilon_2} e_{\varepsilon_2-\varepsilon_4} e_{\varepsilon_1-\varepsilon_3} \in \mathfrak{n}_{-}\mathcal{U}(\mathfrak{g}) + \mathcal{U}(\mathfrak{g}) \mathfrak{n}_{+} \\
		&&
		\left(f_{\varepsilon_2-\varepsilon_4}f_{\varepsilon_1-\varepsilon_3}\right)_L 
		f_{\varepsilon_1-\varepsilon_3}  e_{\varepsilon_1-\varepsilon_3} e_{\varepsilon_2-\varepsilon_4} e_{\varepsilon_1-\varepsilon_3} \in \mathfrak{n}_{-}\mathcal{U}(\mathfrak{g}) + \mathcal{U}(\mathfrak{g}) \mathfrak{n}_{+} \\
		&&
		\left(f_{\varepsilon_2-\varepsilon_4}f_{\varepsilon_1-\varepsilon_3}\right)_L 
		f_{\varepsilon_1-\varepsilon_4}  e_{\varepsilon_1-\varepsilon_4} e_{\varepsilon_2-\varepsilon_4} e_{\varepsilon_1-\varepsilon_3} \in \mathfrak{n}_{-}\mathcal{U}(\mathfrak{g}) + \mathcal{U}(\mathfrak{g}) \mathfrak{n}_{+} \\
		&&
		\left(f_{\varepsilon_2-\varepsilon_4}f_{\varepsilon_1-\varepsilon_3}\right)_L 
		f_{\varepsilon_2-\varepsilon_3}  e_{\varepsilon_2-\varepsilon_3} e_{\varepsilon_2-\varepsilon_4} e_{\varepsilon_1-\varepsilon_3} \in \mathfrak{n}_{-}\mathcal{U}(\mathfrak{g}) + \mathcal{U}(\mathfrak{g}) \mathfrak{n}_{+} \\
		&&
		\left(f_{\varepsilon_2-\varepsilon_4}f_{\varepsilon_1-\varepsilon_3}\right)_L 
		f_{\varepsilon_2-\varepsilon_4}  e_{\varepsilon_2-\varepsilon_4}^2 e_{\varepsilon_1-\varepsilon_3} \in \mathfrak{n}_{-}\mathcal{U}(\mathfrak{g}) + \mathcal{U}(\mathfrak{g}) \mathfrak{n}_{+} \\
		&&
		\left(f_{\varepsilon_2-\varepsilon_4}f_{\varepsilon_1-\varepsilon_3}\right)_L 
		f_{\varepsilon_3-\varepsilon_4}  e_{\varepsilon_3-\varepsilon_4} e_{\varepsilon_2-\varepsilon_4} e_{\varepsilon_1-\varepsilon_3} \in \mathfrak{n}_{-}\mathcal{U}(\mathfrak{g}) + \mathcal{U}(\mathfrak{g}) \mathfrak{n}_{+} \\
		&&
		\left(f_{\varepsilon_2-\varepsilon_4}f_{\varepsilon_1-\varepsilon_3}\right)_L 
		f_{\varepsilon_1-\varepsilon_2}  e_{\varepsilon_1-\varepsilon_2} e_{\varepsilon_2-\varepsilon_3} e_{\varepsilon_1-\varepsilon_4} \in \mathfrak{n}_{-}\mathcal{U}(\mathfrak{g}) + \mathcal{U}(\mathfrak{g}) \mathfrak{n}_{+} \\
		&&
		\left(f_{\varepsilon_2-\varepsilon_4}f_{\varepsilon_1-\varepsilon_3}\right)_L 
		f_{\varepsilon_1-\varepsilon_3}  e_{\varepsilon_1-\varepsilon_3} e_{\varepsilon_2-\varepsilon_3} e_{\varepsilon_1-\varepsilon_4} \in \mathfrak{n}_{-}\mathcal{U}(\mathfrak{g}) + \mathcal{U}(\mathfrak{g}) \mathfrak{n}_{+} \\
		&&
		\left(f_{\varepsilon_2-\varepsilon_4}f_{\varepsilon_1-\varepsilon_3}\right)_L 
		f_{\varepsilon_1-\varepsilon_4}  e_{\varepsilon_1-\varepsilon_4} e_{\varepsilon_2-\varepsilon_3} e_{\varepsilon_1-\varepsilon_4} \in \mathfrak{n}_{-}\mathcal{U}(\mathfrak{g}) + \mathcal{U}(\mathfrak{g}) \mathfrak{n}_{+} \\
		&&
		\left(f_{\varepsilon_2-\varepsilon_4}f_{\varepsilon_1-\varepsilon_3}\right)_L 
		f_{\varepsilon_2-\varepsilon_3}  e_{\varepsilon_2-\varepsilon_3}^2 e_{\varepsilon_1-\varepsilon_4} \in \mathfrak{n}_{-}\mathcal{U}(\mathfrak{g}) + \mathcal{U}(\mathfrak{g}) \mathfrak{n}_{+} \\
		&&
		\left(f_{\varepsilon_2-\varepsilon_4}f_{\varepsilon_1-\varepsilon_3}\right)_L 
		f_{\varepsilon_2-\varepsilon_4}  e_{\varepsilon_2-\varepsilon_4} e_{\varepsilon_2-\varepsilon_3} e_{\varepsilon_1-\varepsilon_4} \in \mathfrak{n}_{-}\mathcal{U}(\mathfrak{g}) + \mathcal{U}(\mathfrak{g}) \mathfrak{n}_{+} \\
		&&
		\left(f_{\varepsilon_2-\varepsilon_4}f_{\varepsilon_1-\varepsilon_3}\right)_L 
		f_{\varepsilon_3-\varepsilon_4}  e_{\varepsilon_3-\varepsilon_4} e_{\varepsilon_2-\varepsilon_3} e_{\varepsilon_1-\varepsilon_4} \in \mathfrak{n}_{-}\mathcal{U}(\mathfrak{g}) + \mathcal{U}(\mathfrak{g}) \mathfrak{n}_{+} \\
		&&
		\left(f_{\varepsilon_2-\varepsilon_4}f_{\varepsilon_1-\varepsilon_3}\right)_L 
		h_3  f_{\varepsilon_1-\varepsilon_2} e_{\varepsilon_1-\varepsilon_4} e_{\varepsilon_1-\varepsilon_3} \in \mathfrak{n}_{-}\mathcal{U}(\mathfrak{g}) + \mathcal{U}(\mathfrak{g}) \mathfrak{n}_{+} \\
		&&
		\left(f_{\varepsilon_2-\varepsilon_4}f_{\varepsilon_1-\varepsilon_3}\right)_L 
		f_{\varepsilon_1-\varepsilon_2} e_{\varepsilon_3-\varepsilon_4} e_{\varepsilon_1-\varepsilon_3}^2 \in \mathfrak{n}_{-}\mathcal{U}(\mathfrak{g}) + \mathcal{U}(\mathfrak{g}) \mathfrak{n}_{+} \\
		&&
		\left(f_{\varepsilon_2-\varepsilon_4}f_{\varepsilon_1-\varepsilon_3}\right)_L 
		h_1  f_{\varepsilon_3-\varepsilon_4} e_{\varepsilon_2-\varepsilon_4} e_{\varepsilon_1-\varepsilon_4} \in \mathfrak{n}_{-}\mathcal{U}(\mathfrak{g}) + \mathcal{U}(\mathfrak{g}) \mathfrak{n}_{+} \\
		&&
		\left(f_{\varepsilon_2-\varepsilon_4}f_{\varepsilon_1-\varepsilon_3}\right)_L 
		f_{\varepsilon_3-\varepsilon_4}  e_{\varepsilon_1-\varepsilon_2} e_{\varepsilon_2-\varepsilon_4}^2 \in \mathfrak{n}_{-}\mathcal{U}(\mathfrak{g}) + \mathcal{U}(\mathfrak{g}) \mathfrak{n}_{+} \\
		&&
		\left(f_{\varepsilon_2-\varepsilon_4}f_{\varepsilon_1-\varepsilon_3}\right)_L 
		f_{\varepsilon_3-\varepsilon_4}  f_{\varepsilon_1-\varepsilon_2} e_{\varepsilon_1-\varepsilon_4}^2  \in \mathfrak{n}_{-}\mathcal{U}(\mathfrak{g}) + \mathcal{U}(\mathfrak{g}) \mathfrak{n}_{+}. 
	\end{eqnarray*}
Using the explicit formula for $v'$ from Proposition \ref{Zhu-algebra}, one easily obtains that  	
\begin{equation*} 
	\begin{aligned}
		& v_1 \in -\frac{5}{2}h_2 - \frac{7}{2}h_1h_2 + \frac{3}{2}h_1h_3 - \frac{7}{2}h_2h_3 - \frac{31}{6}{h_2}^2 - \frac{13}{3}h_1{h_2}^2 - {h_1}^2h_2 - 2h_1h_2h_3 -\\
		&- \frac{10}{3}{h_2}^3 + h_1{h_3}^2 + {h_1}^2h_3 - \frac{13}{3}{h_2}^2h_3 - h_2{h_3}^2 - \frac{4}{3}h_1{h_2}^3 - \frac{2}{3}{h_1}^2{h_2}^2 +\\
		&+ \frac{2}{3}{h_1}^2{h_3}^2 - \frac{4}{3}{h_1}{h_2}^2h_3 - \frac{2}{3}{h_2}^4 - \frac{4}{3}{h_2}^3h_3 - \frac{2}{3}{h_2}^2{h_3}^2 + \mathfrak{n}_{-}\mathcal{U}(\mathfrak{g}) + \mathcal{U}(\mathfrak{g}) \mathfrak{n}_{+},
	\end{aligned}
\end{equation*}
which implies that $p_1 \in {\mathcal P}_{0}$. Similarly, one obtains that 
$$ v_2 \in p_2(h) + \mathfrak{n}_{-}\mathcal{U}(\mathfrak{g}) + \mathcal{U}(\mathfrak{g}) \mathfrak{n}_{+},
$$
which implies that $p_2 \in {\mathcal P}_{0}$. Since $p_1$ and $p_2$ are linearly independent, we obtain the claim of Lemma.
\end{proof}

The following proposition gives the classification of irreducible $A( \widetilde{L}_{-5/2}( \mathfrak{g}))$--modules in the category $\mathcal{O}$:

\begin{proposition}\label{Zhu-moduli}
	The complete list of irreducible $A( \widetilde{L}_{-5/2}( \mathfrak{g}))$--modules in the category $\mathcal{O}$ is given by the set:
	\begin{equation} \label{moduli-kat-O-skup}
	\left\lbrace V(\mu_i(t) ) \ | \ i=1,\ldots,16, \ t\in \mathbb{C}\right\rbrace,
	\end{equation}
	 where:\\
	\begin{tabular}{ p{6cm} p{9cm} }
		$\mu_1(t)=t\omega_1$, & $\mu_9(t)=-\frac{3}{2}\omega_1 + t\omega_3$, \vspace*{0.2cm} \\ 
		$\mu_2(t)=t\omega_3$, & 	$\mu_{10}(t)=t\omega_1-\frac{3}{2}\omega_3 $, \vspace*{0.2cm} \\  
		$\mu_3(t)=t\omega_1+(-t-\frac{5}{2} ) \omega_{2}$, & 	$\mu_{11}(t)=-\frac{3}{2}\omega_1+t\omega_2+\left(-t-1 \right) \omega_3$,  \vspace*{0.2cm}\\ 
		$\mu_4(t)=t\omega_2+(-t-\frac{5}{2} ) \omega_{3}$, & 	$\mu_{12}(t)=\left(-t-1 \right) \omega_1 + t\omega_2-\frac{3}{2}\omega_3$, \vspace*{0.2cm} \\ 
		$\mu_5(t)=t\omega_1 - \frac{3}{2}\omega_2$, & 	$\mu_{13}(t)=-\frac{1}{2}\omega_1-\frac{1}{2}\omega_2+t\omega_3$, \vspace*{0.2cm} \\ 
		$\mu_6(t)=- \frac{3}{2}\omega_2 +t\omega_3$, & $\mu_{14}(t)=-\frac{1}{2}\omega_1+t\omega_2+(-t-\frac{3}{2} )\omega_3$,  \vspace*{0.2cm}\\ 
		$\mu_7(t)=t\omega_1+(-t-1 ) \omega_{2}$, & $\mu_{15}(t)=t\omega_1-\frac{1}{2}\omega_2-\frac{1}{2}\omega_3$,  \vspace*{0.2cm}\\ 
		$\mu_8(t)=t\omega_2+(-t-1 ) \omega_{3}$, & $\mu_{16}(t)=(-t-\frac{3}{2}) \omega_1+t\omega_2-\frac{1}{2}\omega_3.$   
	\end{tabular}
\end{proposition}
\begin{proof} 
Let $V(\mu )$, for $\mu \in \mathfrak{h}^*$, be an irreducible $A( \widetilde{L}_{-5/2}( \mathfrak{g}))$--module in the category 
$\mathcal{O}$. Then Corollary \ref{koro-polinomi} and Lemma \ref{lema-polinomi} imply that the highest weight $\mu$ is a solution of polynomial equations
$$ p_1(\mu(h))= p_2(\mu(h))=0. $$
Let us denote $H_i=\mu(h_i ) $, for $i=1,2,3$. Relation $p_2(\mu(h))=0$ now gives 
\begin{equation} \label{f1}
H_2 \left( H_1+H_2+H_3+\frac{5}{2}\right) \underbrace{\left( 3+3H_1+5H_2+2H_1H_2+2{H_2}^2+3H_3+4H_1H_3+2H_2H_3 \right)}_\text{$Q_2$}   =0,
\end{equation}
and relation $p_1(\mu(h))+ p_2(\mu(h))=0$ gives:
\begin{equation} \label{f2}
H_1H_3 \underbrace{\left( 9+6H_1+20H_2+8H_1H_2+8{H_2}^2+6H_3+4H_1H_3+8H_2H_3\right)}_\text{$Q_1$} =0.
\end{equation}

From relations (\ref{f1}) and (\ref{f2}) we conclude:
\begin{itemize}
	\item For $H_1=0$ and $H_2=0$, we obtain $\mu=t\omega_3$.
	\item For $H_2=0$ and $H_3=0$, we obtain $\mu=t\omega_1$.
	\item For $H_1=0$ and $H_1+H_2+H_3+\frac{5}{2}=0$, we obtain $\mu=t\omega_2+\left(-t-\frac{5}{2} \right)\omega_3 $.
	\item For $H_3=0$ and $H_1+H_2+H_3+\frac{5}{2}=0$, we obtain $\mu=\left(-t-\frac{5}{2} \right)\omega_1+t\omega_2 $.
	\item For $H_1=0$ and $Q_2=0$, we have $\left(3+2H_2\right) \left( H_2+H_3+1\right) =0 $, so we obtain $\mu=-\frac{3}{2}\omega_2+t\omega_3$ or $\mu=t\omega_2+\left(-t-1 \right)\omega_3 $. 
	\item For $H_3=0$ and $Q_2=0$, we have $\left(3+2H_2\right) \left( H_1+H_2+1\right) =0 $, so we obtain $\mu=t\omega_1-\frac{3}{2}\omega_2$ or $\mu=\left(-t-1 \right)\omega_1+ t\omega_2$.
	\item For $Q_1=0$ and $H_2=0$, we have $\left(3+2H_1\right) \left( 3+2H_3\right) =0 $, so we obtain $\mu=-\frac{3}{2}\omega_1+t\omega_3$ or \\ $\mu=t\omega_1-\frac{3}{2}\omega_3$.
	\item For $Q_1=0$ and $H_1+H_2+H_3+\frac{5}{2}=0$, we have $\left(3+2H_1\right) \left( 3+2H_3\right) =0 $, so we obtain \\$\mu=-\frac{3}{2}\omega_1+t\omega_2+\left(-t-1 \right) \omega_3 $ or $\mu=\left(-t-1 \right) \omega_1 + t\omega_2-\frac{3}{2}\omega_3 $.
	\item For $Q_1=0$ and $Q_2=0$, we have 
	$$Q_1-4Q_2=-12H_1H_3-6H_1-6H_3-3=0$$
	 which implies $\left(1+2H_1 \right) \left(1+2H_3 \right) =0$. If $H_1=-\frac{1}{2}$, then $Q_1=0$ implies that
	 $$\left(1+2H_2 \right) \left( H_2+H_3+\frac{3}{2}\right)=0,$$
	 so we obtain $\mu=-\frac{1}{2}\omega_1-\frac{1}{2}\omega_2+t\omega_3$ or $\mu=-\frac{1}{2}\omega_1+t\omega_2+\left(-t-\frac{3}{2} \right)\omega_3 $. Similarly, if $H_3=-\frac{1}{2}$, then $Q_1=0$ implies
	  $$\left(1+2H_2 \right) \left( H_1+H_2+\frac{3}{2} \right)=0,$$
	 so we obtain $\mu=t\omega_1-\frac{1}{2}\omega_2-\frac{1}{2}\omega_3$ or $\mu=\left(-t-\frac{3}{2} \right)\omega_1+t\omega_2-\frac{1}{2}\omega_3 $.
\end{itemize}
The proof now follows from Corollary \ref{koro-polinomi}.
\end{proof}

Note that the weights $\mu_i(t)$ from the set (\ref{moduli-kat-O-skup}) are dominant integral only for $i=1,2$ and $t \in \mathbb {Z}_{\ge 0}$. Thus, we obtain:

\begin{corollary}\label{Korolar 1}
	The set
	$$ \{V( t \omega_1) \ | \  t \in \mathbb {Z}_{\ge 0} \} \cup \{ V(t \omega_3)\ | \ t \in \mathbb {Z}_{\ge 0} \}$$
	provides a complete list of irreducible finite-dimensional modules for the Zhu's algebra $A( \widetilde{L}_{-5/2}( \mathfrak{g}))$.
\end{corollary}

Zhu's theory now implies:

\begin{theorem} \label{kvocijent-moduli} Using the notation from Proposition \ref{Zhu-moduli}, the set 
	$$ \left\lbrace L_{-5/2}(\mu_i(t)) \ | \ i=1,\ldots,16, \ t\in \mathbb{C}\right\rbrace $$
	provides a complete list of irreducible $\widetilde{L}_{-5/2}\left( \mathfrak{g}\right)$--modules in the category $\mathcal{O}$.
\end{theorem}

\begin{corollary} \label{klas-quot-KL} The set 
	$$ \left\lbrace L_{-5/2}( t \omega_1 )\ | \ t\in \mathbb {Z}_{\ge 0}\right\rbrace \cup \left\lbrace L_{-5/2}( t \omega_3 )\ | \ t\in \mathbb {Z}_{\ge 0}\right\rbrace $$
	provides a complete list of irreducible $\widetilde{L}_{-5/2}( \mathfrak{g})$--modules in the category $KL_{-5/2}$.
\end{corollary}

\begin{remark} \label{rem-klas-simple}
In what follows, we will prove that the vertex algebra $\widetilde{L}_{-5/2}( \mathfrak{g})$ is simple (see Section \ref{sect-collapsing}), so Theorem~\ref{kvocijent-moduli} and Corollary \ref{klas-quot-KL} will give classifications of irreducible $L_{-5/2}\left( \mathfrak{g}\right)$--modules in the category $\mathcal{O}$, and irreducible $L_{-5/2}( \mathfrak{g})$--modules in the category $KL_{-5/2}$, respectively. The crucial part of the proof of that simplicity result is the classification of irreducible $\widetilde{L}_{-5/2}( \mathfrak{g})$--modules in the category $KL_{-5/2}$ from Corollary \ref{klas-quot-KL}.
\end{remark}

\section{The description of maximal ideal in $V^{-5/2}(sl(4))$ and semi-simplicity of $KL_{-5/2}$} \label{sect-collapsing}

In this section we prove that the vertex algebra $\widetilde{L}_{-5/2}( \g)$, defined in Section \ref{sect-singular-vector}, is simple, i.e. that the ideal $\langle v \rangle$ in $V^{-5/2}(\g)$ generated by the singular vector $v$ from Theorem \ref{thm-singv} is the maximal ideal in $V^{-5/2}(\g)$. Furthermore, we prove that the category $KL_{-5/2}$ for $L_{-5/2}( \g)$ is semi-simple. Throughout this section and Section \ref{singular-new} we will use the notation
$$J^{-5/2} := \langle v \rangle,$$
i.e. $\widetilde{L}_{-5/2}( \g) = V^{-5/2}(\g) / J^{-5/2}$.

\subsection{Maximal ideal in $V^{-1}(sl(4))$  }
\label{review-1}
For a comparison, we recall the description of the maximal ideal in the case $k=-1$ which was obtained by T. Arakawa and A. Moreau.
\begin{itemize}
	\item Level $k=-1$ is collapsing, and the corresponding minimal affine $\mathcal{W}$--algebra
	$W_{-1}(sl(4), f_{\theta})$ collapses to the Heisenberg vertex algebra $M(1)$. This was proved in two different ways: in \cite{ArMo-sheets} using singular vectors in $V^{-1}(sl(4))$ (obtained in \cite{AP-2007}), and in \cite{AKMPP-JA} by proving that the $G$-generators  of $W^{-1}(sl(4), f_{\theta})$ of conformal weight $3/2$  belong to the maximal ideal.
	\item By applying Hamiltonian reduction, we see that the ideal  $J^{-1}$ generated by a singular vector from \cite{AP-2007} is mapped to an ideal in $W^{-1}(sl(4), f_{\theta})$ containing all $G$-generators. But from $G$-generators one can reconstruct $sl(2)$--part of $\g^{\natural} = gl(2)$. Therefore, the Hamiltonian reduction maps ideal $J^{-1}$ to the maximal ideal in $W^{-1}(sl(4), f_{\theta} )$, i.e., $H_{f_{\theta}} (V^{-1} (sl(4)) / J^{-1} ) = M(1)$.
	
	\item Assuming  that $ V^{-1} (sl(4)) / J^{-1} $ is not simple, using the properties of the functor $H_{f_{\theta}}$ we get $H_{f_{\theta}} (L_{-1}(sl(4))) = 0$, which is a contradiction.
	
	\item Therefore, $ J^{-1}$ must be maximal ideal in $V^{-1}(sl(4))$.
	
	\item It is proved later in \cite{AKMPP-IMRN} that $KL_{-1}$ is a semi-simple category, and in \cite{CY} that $KL_{-1}$ is a braided tensor category.	
	
\end{itemize}

\subsection{An approach in the case  $V^{-5/2}(sl(4))$  }
In the case $k=-5/2$ for $\mathfrak{g}=sl(4)$ we apply similar approach as in the case $k=-1$. But instead of using minimal affine $\mathcal{W}$--algebra $W^k(\g, f_{\theta})$, we shall use the  affine $\mathcal{W}$--algebra $W^k(\g, f_{subreg})$ associated to the subregular nilpotent element, defined in Subsection~\ref{subsect-W-subreg}.

Vertex algebra $W^k(\g, f_{subreg})$ is generated by $J, L = \bar{L}- \partial J, W, G^+, G^-$ having conformal weights  $1,2,3,2,2$. The OPE is presented by T. Creutzig and A. Linshaw in \cite{CL} (see also \cite{G-2017}). We recall the OPE for $W^k(\g, f_{subreg})$ in the Appendix.

\begin{theorem} \label{collapsing-1}
	Level $k=-5/2$ is a collapsing level for $W^k(\g, f_{subreg})$  and 
	$$W_{-5/2} (\g, f_{subreg}) \cong M_J(1),$$
	where $M_J(1)$ is the Heisenberg vertex algebra generated by $J$.
\end{theorem}
\begin{proof}  
	The method of the proof  uses  concepts developed in \cite{AKMPP-JA}  for  collapsing levels. We need to prove that $G^{\pm}$ belong to the maximal ideal of $W^{-5/2}(\g, f_{subreg})$.
	Using OPE for $W^k(\g, f_{subreg})$ in the case $k=-5/2$ we conclude:
	$$ G^+ _3 G^- = G^+ _2 G^- = 0, \  G^+ _1  G^- \sim  ( L   -  4  : J J:).  $$
	Together with other OPE relations, this implies that $G^{\pm}$ belong to the maximal ideal in \linebreak $W^{-5/2}(\g, f_{subreg})$, i.e.  $G^{\pm} = 0$ in  $W_{-5/2} (\g, f_{subreg})$. As a consequence we get $L= 4 :J J:$, implying that $M_J(1)$ is conformally embedded into $W_{-5/2} (\g, f_{subreg})$.
	
	It remains to show that $W \in M_J(1)$.
	
	Applying OPE relations, and using the fact that $G^+ _0 G^- = 0$ and $L= 4 : J J:$, we get that $W \in M_J(1) \subset W_{-5/2} (\g, f_{subreg})$.
	In this way we have proved  that  $W_{-5/2} (\g, f_{subreg})$ is generated only with the Heisenberg field, so it is isomorphic to $M_J(1)$. The proof follows.
\end{proof}

The proof of the following lemma is similar to the proof of Lemma 7.3. from 
\cite{ArMo-sheets}: 
\begin{lemma} \label{lem-img-sing}
	The image of singular vector $v$ from Theorem \ref{thm-singv} in 
	$W^{-5/2}(\g, f_{subreg})$ coincides (up to a non-zero scalar) with the vector $G^+$.
\end{lemma}
\begin{proof} Since $v$ is singular in $V^{-5/2}( \mathfrak{g})$, it is mapped to  a singular vector $\tilde v$ of $W^{-5/2}(\g, f_{subreg})$. It can be shown that, if $\tilde v$ is non-zero, it has conformal weight $1$ with respect to $\bar  L$.
	Its image in $R_{W^{-5/2}(\g, f_{subreg})}$ is the image of vector $v''$ from Proposition \ref{C2-algebra} in 
	$\left( \mathcal{S}(\g)/ J_\chi  \right) ^M$, where $\chi=\left( f_{subreg} \, | \, \cdot \right)$ (see Proposition \ref{prop-C2-W}). It follows from Lemma \ref{lem-singv-modJ} that $v'' \equiv 2e_{\varepsilon_1-\varepsilon_2} (\mathrm{mod} \ J_\chi)$, which implies that $\tilde v \equiv 2 G^+ (\mathrm{mod} \ C_2(W^{-5/2}(\g, f_{subreg})))$. This implies that  $\tilde v$ is non-zero, and  $\tilde v$ and $2G^+$ must coincide since they have the same conformal weight and they  are  both singular vectors for $M_J(1)$.
\end{proof}
\begin{proposition} \label{prop-Hf-L-tilde}
	We have:
	\item[(1)] $H_{f_{subreg}} (\widetilde{L}_{-5/2}( \mathfrak{g}) ) \cong M_J(1).$
	\item[(2)] 
	$H_{f_{subreg}} ({L}_{-5/2}( \mathfrak{g}) ) \cong M_J(1).$
\end{proposition}
\begin{proof}
	Since $J^{-5/2} = \langle v \rangle$, Lemma \ref{lem-img-sing} implies that $H_{f_{subreg}} (J^{-5/2})$ contains generator $G^+$ of \linebreak $W^{-5/2} (\g, f_{subreg})$. As in the proof of Theorem \ref{collapsing-1} we conclude that $L   -  4  : J J: \ \in H_{f_{subreg}} (J^{-5/2})$. This implies that
	$$(L   -  4  : J J:)_1 G^- = -2 G^- \in H_{f_{subreg}} (J^{-5/2}).$$
	As in the proof of Theorem \ref{collapsing-1}  we conclude that  $H_{f_{subreg}} (J^{-5/2})$ coincides with the maximal ideal in $W^{-5/2} (\g, f_{subreg})$. Using exactness of QHR functor $H_{f_{subreg}}( \, \cdot \,)$ on the category $KL_{-5/2}$ (cf. \cite{Ara-IMRN}), we get that
	$H_{f_{subreg}} (\widetilde{L}_{-5/2}( \mathfrak{g}) )  = W_{-5/2} (\g, f_{subreg}) = M_J (1)$.
	
	This proves assertion (1).
	
	Using again exactness of the functor $H_{f_{subreg}}( \, \cdot \,)$ we see that
	$H_{f_{subreg}} ({L}_{-5/2}( \mathfrak{g}) ) \cong M_J(1) $ or $H_{f_{subreg}} ({L}_{-5/2}( \mathfrak{g}))  = \{0\}$.
	
	Assume that $H_{f_{subreg}} ({L}_{-5/2}( \mathfrak{g}))  = \{0\}$.  Then $ \widetilde{L}_{-5/2}( \mathfrak{g})$ must contain a singular vector $w_{\mu}$ which is mapped to ${\bf 1} \in M_J(1)$.
	
	The classification result from Corollary \ref{klas-quot-KL} implies that the $\g$--weight of $w_{\mu}$ is $\mu = n \omega_1$ or $\mu= n \omega_3$  for $n >0$. Direct calculation shows that $w_{\mu}$ is a mapped to a highest weight vector in $W_{-5/2} (\g, f_{subreg})$  of $\bar{L}(0)$ weight $\frac{n(n+1)}{4}$ (for $\mu = n \omega_1$)  or $\frac{n(n-1)}{4}$ (for $\mu = n\omega_3$).  This implies that in the case $\mu = n\omega_1$ we have $n=0$. A contradiction. In the case 
	$\mu = n\omega_3$ we have $n=0$ or $n=1$. The case $n =0$ again leads to a contradiction, and the case $n=1$ is also not  possible since then conformal weight of $w_{\mu}$ is equal to $5/4 \notin {\Bbb Z}$. This proves the assertion (2).

\end{proof}

Next result gives some fundamental properties of the functor  $H_{f_{subreg}}$:

\begin{theorem} \label{conj-p} For any   $n \in {\Bbb Z}_{> 0}$ we have:
	
	\item[(P)]  $H_{f_{subreg}} ({L}_{-5/2} (n \omega_3)) \ne \{ 0 \}$   and
	$H_{f_{subreg}} (M ) = \{0\}$    for any highest weight $\widetilde{L}_{-5/2}(\mathfrak{g})$--module $M$  in $KL_{-5/2}$ of $\mathfrak{g}$--weight $n \omega_1$. 
\end{theorem}

The proof of Theorem \ref{conj-p} will be given in Section \ref{singular-new}.

\begin{remark} Property (P) from the Theorem \ref{conj-p} is a major difference between the case $\g =sl(4)$ and $k =-5/2$, which we study in the present paper, and the case $\g =sl(4)$ and $k =-1$ presented in Subsection~\ref{review-1}, treated using the minimal reduction functor~$H_{f_{\theta}}$. In the case $k =-1$ we have:
	$$H_{f_{\theta}}({L}_{-1} (n \omega_1)) \ne \{ 0 \} \quad \mbox{and} \quad H_{f_{\theta}}({L}_{-1} (n \omega_3))\ne \{ 0 \}, $$
	and these properties are a very important step in determining the maximal ideal in $V^{-1} (sl(4))$. The idea in the case $k =-5/2$ is to use properties from Theorem \ref{conj-p}, along with the automorphism $\sigma$ from Remark \ref{automorphism}
	that interchanges the weights $n \omega_1$ and $n \omega_3$.
\end{remark}

Theorem \ref{conj-p} gives us an important technical result for proving the following structural result for the simple vertex algebra $L_{-5/2}\left( \mathfrak{g}\right)$ and its module category $KL_{-5/2}$:

\begin{theorem}   \label{main-max}  We have:
	\begin{itemize}
		\item[(i)] $J^{-5/2}$ is the maximal ideal in $V^{-5/2}(sl(4))$, i.e. $L_{-5/2}\left( sl(4)\right) \cong V^{-5/2}(sl(4))/ J^{-5/2} $.
		\item[(ii)] The category $KL_{-5/2}$ is semi-simple.
	\end{itemize}
	
\end{theorem}
\begin{proof}
	Proposition \ref{prop-Hf-L-tilde} gives that $H_{f_{subreg}} (\widetilde{L}_{-5/2}( \mathfrak{g}) )  = W_{-5/2} (\g, f_{subreg}) = M_J (1)$. If $\widetilde{L}_{-5/2}( \mathfrak{g}) $ is not simple, the classification result from Corollary \ref{klas-quot-KL} implies that it must contain a non-trivial singular vector $w_{\mu}$  of $\mathfrak{g}$--weight 
	$\mu=n\omega_1$ or $\mu=n \omega_3$, for $n\in \Bbb{Z}_{> 0}$.
	Let $\sigma$ be the automorphism of $V^{-5/2} (\mathfrak{g})$ of order two from Remark \ref{automorphism}.
	
	Since $\sigma$ fixes singular vector $v$, it is also an automorphism of $\widetilde{L}_{-5/2}(\mathfrak{g})$ and it must map  any (possible) singular vector of weight $n \omega_1$ to a singular vector of weight $n\omega_3$. We thus conclude that   $\widetilde{L}_{-5/2}(\mathfrak{g})$ contains a singular vector of weight $\mu = n \omega_3$ for $n >0$. Theorem \ref{conj-p} then implies that  the ideal generated by this singular vector is mapped  by the QHR functor $H_{f_{subreg}}$ to an non-trivial ideal in the vertex algebra $ W_{-5/2} (\g, f_{subreg})$, which is simple. By using exactness of  the functor $H_{f_{subreg}}$,  we get   that  $H_{f_{subreg}} (L_{-5/2} (\g) ) = \{0\}$,   which is a contradiction with Proposition \ref{prop-Hf-L-tilde} (2). This proves (i), i.e. $L_{-5/2}( \mathfrak{g}) \cong \widetilde{L}_{-5/2}( \mathfrak{g})$.
	
	The similar  arguments as in (i) show that any highest weight module in $KL_{-5/2}$ must be irreducible. Let us prove this statement. Consider the
	highest weight module $L_{-5/2} (\mathfrak{g})$--module $M(\mu)$ of the $\mathfrak{g}$--weight $\mu = n \omega_1$ or $\mu=n \omega_3 $ for $n \in {\Bbb Z}_{>0}$. Assume that $M(\mu)$ is not irreducible. Then there is a singular vector $z_{\nu} \in M(\mu)$ of highest weight $\nu = m \omega_3$ or $\nu = m \omega_1$, where $m \in {\Bbb Z}, m > n$.  Denote by $Z({\nu}) = L_{-5/2} (\mathfrak{g}). z_{\nu}$.
	By applying QHR functor $H_{f_{subreg}}( \, \cdot \,)$  we shall see that all cases lead to a contradiction to Theorem \ref{conj-p}.
	
	\begin{itemize}
		\item[(1)]  The case $\mu = n\omega_3$, $\nu = m \omega_3$. We get 
		$   H_{f_{subreg}}(M(\mu) ) = H_{f_{subreg}}(Z(\nu) )  = M_J(1, a)$ for certain $a \in {\Bbb C}$. This implies $H_{f_{subreg}} ({L}_{-5/2} (n \omega_3)) = \{0\}$. A contradiction.
		
		\item[(2)]  The case $\mu = n\omega_1$, $\nu = m\omega_3$. We get 
		$   H_{f_{subreg}}(M(\mu) ) = \{0\},  H_{f_{subreg}}(Z(\nu) )  = M_J(1, a)$ for certain $a \in {\Bbb C}$.  This implies that  $H_{f_{subreg}}( \, \cdot \,)$  sends  exact sequence
		$$ 0 \rightarrow Z(\nu) \rightarrow M({\mu})  \rightarrow  M({\mu}) / Z(\nu) \rightarrow 0 $$
		to the non-exact sequence 
		$$ 0 \rightarrow  M_J (1, a) \rightarrow 0 \rightarrow  0 \rightarrow 0. $$
		A contradiction.
		
		\item[(3)] The case $\mu = n\omega_1$, $\nu = m\omega_1$. Using  (1)  and applying the automorphism $\sigma$ we get a contradiction.

		\item[(4)] The case $\mu = n\omega_3$, $\nu = m\omega_1$. Using  (2)  and applying the automorphism $\sigma$ we get a contradiction.

	\end{itemize}

	Recall that \cite[Theorem 5.5]{AKMPP-IMRN} says that if any highest weight $L_k(\g)$--module $M$ in $KL_k$ is irreducible, then the category $KL_k$ is semi-simple. So $KL_{-5/2}$ is semi-simple.

	 \end{proof}

\section{Singular vectors in $V^k (n \omega_1)$ and $V^k(n \omega_3)$ and the proof of Theorem \ref{conj-p}}
\label{singular-new}

In this section we give a proof of Theorem \ref{conj-p}.  The main method in the proof is based on a construction of singular vectors in generalized Verma modules $V^k (n \omega_i)$, $i=1,3$, and the description of their submodules $J^k \cdot V^k(n \omega_i)$. As a consequence, we construct universal $\widetilde L_k(sl(4))$--modules $\overline M(n \omega_i)$, for which we prove vanishing and non-vanishing of    $H_{f_{subreg}} ( \overline M(n \omega_i))$.

\subsection{The universal $\widetilde L_{-5/2}(\g)$--modules 
$\overline M(n\omega_1)$ and  $\overline M(n\omega_3)$}For an irreducible $\g =sl(4)$--module $V(\mu)$, we define the generalized Verma module of level $k$:
$$ V^k(\mu):=U(\hat{\g}) \otimes_{U(\g \otimes {\Bbb C}[t]+ CK)} V(\mu). $$

Then, $V^k(\mu)$ is a ${\Bbb Z}_{\ge 0}$--graded module:
$$ V^k(\mu)= \bigoplus_{m \in {\Bbb Z}_{\ge 0}}  V^k(\mu) (m),$$
$$ L(0) \vert V^k(\mu) (m) \equiv \left (m+ \frac{ (\mu \, | \, \mu + 2 \rho)}{2 (k+ h^{\vee})} \right) \ \mbox{id}.$$
For $w \in V^k(\mu) (m)$, we write $\deg(w) = m$. Let us denote by $v_{\mu}$ the highest weight vector of $V^k(\mu)$.

Let $k =-\frac{5}{2}$. Recall that $V^k(\g)$ contains  singular vector  $v$ of $\g$--weight $2 \omega_2$ and conformal weight~$4$. So, we have a non-trivial $\hat{\g}$--homomorphism:
$$ V^k (2 \omega_2) \rightarrow V^k(\g). $$

 Next we are interested in a construction of the universal  $\widetilde L_{-5/2}(\g)$--module $\overline M(\mu)$ in $KL_{-5/2}$, for highest weights $\mu = n \omega_1$ and $\mu = n \omega_3$. It is realized as
$$ \overline M(\mu) = \frac{V^k(\mu)}{ J^k \cdot V^k(\mu)}, $$
where 
$J^k \cdot V^k(\mu ) =\mbox{span} \{ a_n w \ \vert \ a \in J^k, w \in V^k(\mu), n \in \Bbb Z\}$ is a $V^k(\g)$--submodule of  $V^k(\mu)$ generated by the action of $J^k$. Since the top component of $V^k(\mu)$ does not belong to  $J^k \cdot V^k(\mu )$, we get that   $J^k \cdot V^k(\mu )$ is a proper submodule of  $V^k(\mu )$. So, module  $\overline M(\mu)$ is not zero, and $L_k(\mu)$ is its simple quotient.

\begin{lemma} Let $n \in {\Bbb Z}_{>0}$. We have:
	\item[(1)] As a $V^k(\g)$--module,  $J ^k  \cdot  V^k(n \omega_1)$  is generated by singular vectors $w_{\nu}$ of $\g$--weight  $\nu$ for some
	$$ \nu \in \{ n \omega_1 + 2 \omega_2, ( n-1) \omega_1 + \omega_2 + \omega_3, ( n-2) \omega_1 + 2\omega_3   \}. $$
	
	\item[(2)]  As a $V^k(\g)$--module, $J^k \cdot  V^k(n \omega_3)$ is generated by singular vectors  $w_{\nu}$ of $\g$--weight $\nu$ for some
	$$ \nu \in \{ n \omega_3 + 2 \omega_2, ( n-1) \omega_3 + \omega_1 + \omega_2, ( n-2) \omega_3 + 2\omega_1   \}. $$
\end{lemma}
 \begin{proof}
 
 Assume that $M_1$ and $M_2$ are $V^k (\g)$--modules in $KL^k$ such that there exist surjective intertwining operators of types
 $$ { M_1 \choose J^k \ \ V^k(n \omega_1) }, \quad { M_2 \choose J^k \ \ V^k(n \omega_3) }. $$
 Then by using the fusion rules arguments  and the following decompositions of  tensor products of $\g$--modules:
$$ V(2 \omega_2) \otimes V(n \omega_1) = V (n \omega_1 + 2 \omega_2) \oplus
V (( n-1) \omega_1 + \omega_2 + \omega_3) \oplus V (( n-2) \omega_1 + 2\omega_3), $$
$$ V(2 \omega_2) \otimes V(n \omega_3) = V (n \omega_3 + 2 \omega_2) \oplus
V (( n-1) \omega_3 + \omega_1 + \omega_2) \oplus V (( n-2) \omega_3 + 2\omega_1), $$
we get that in the category $V^k(\g)$--modules we have 
\bea
M_1 &\subset& \widetilde V ^k(n \omega_1 + 2 \omega_2) +
\widetilde V ^k (( n-1) \omega_1 + \omega_2 + \omega_3) +  \widetilde V^k (( n-2) \omega_1 + 2\omega_3) \nonumber \\
M_2  &\subset& \widetilde V ^k(n \omega_3 + 2 \omega_2) +
 \widetilde V ^k (( n-1) \omega_3 + \omega_1 + \omega_2) +  \widetilde  V^k (( n-2) \omega_3 + 2\omega_1), \nonumber
\eea 
where $\widetilde V^k(\lambda)$ denotes certain quotient of $V^k(\lambda)$.

Since the restriction of the vertex operator on $V^k (\mu)$ gives a non-trivial surjective intertwining operator of type
$$ { J^k \cdot V^k(\mu) \choose J^k \ \ V^k(\mu) },$$
we have proved  that there exist nontrivial surjective  homomorphisms
 \bea
  V ^k(n \omega_1 + 2 \omega_2) +
  V ^k (( n-1) \omega_1 + \omega_2 + \omega_3) +    V^k (( n-2) \omega_1 + 2\omega_3) &\rightarrow& J^k \cdot V^k(n \omega_1), \nonumber \\
 V ^k(n \omega_3 + 2 \omega_2) +
  V ^k (( n-1) \omega_3 + \omega_1 + \omega_2) +   V^k (( n-2) \omega_3 + 2\omega_1)& \rightarrow& J^k \cdot V^k(n \omega_3). \nonumber
\eea
The proof follows.
\end{proof}

\begin{lemma} \label{sing-1-dod}
	Let $\mu = n \omega_1$ or $\mu = n \omega_3$, $n \in {\Bbb Z}_{>0}$.  Then  $J^k \cdot V^k(\mu)$ contains a non-trivial singular vector $w_{\nu}$ such that $\deg(w_{\nu}) = 2$  
	and whose $\g$--weight is equal to  $\nu  = (n-1) \omega_1 +  \omega_2 + \omega_3$  (for $\mu = n \omega_1$) or  $\nu  = (n-1) \omega_3 + \omega_1+  \omega_2$ (for $\mu = n \omega_3$). 
\end{lemma}
\begin{proof}
	We see directly that $v_{-1} v_{\mu}$ is a non-zero element of degree $4$ of the  submodule   $J^k \cdot V^k(\mu)$. Therefore, it  must contain a singular vector of degree $1$, $2$, $3$ or $4$. 
	
	Assuming that there is a singular vector $w_{\nu}$ of certain $\g$--weight $\nu$, one can easily calculate degrees of the singular vectors:

\begin{itemize}
	\item $\nu  = n \omega_1 + 2 \omega_2$ (resp. $\nu  = n \omega_3 + 2 \omega_2$), then in  $V^k(n \omega_1) $ (resp. $V^k(n \omega_3) $),  $\deg (w_{\nu}) = 4 + \frac{2}{3} n$.
	\item $\nu  = (n-1) \omega_1 +  \omega_2 + \omega_3$ (resp. $\nu  = (n-1) \omega_3 + \omega_1+  \omega_2$), then in  $V^k(n \omega_1) $ (resp. $V^k(n \omega_3) $),  $\deg (w_{\nu}) = 2$.
	\item $\nu  = (n-2) \omega_1 + 2 \omega_3$ (resp. $\nu  = (n-2) \omega_3 + 2 \omega_1$), then in  $V^k(n \omega_1) $ (resp. $V^k(n \omega_3) $),  $\deg (w_{\nu}) =\frac{2\left( 2-n \right) }{3} $.
\end{itemize}

This implies that $J^k \cdot V^k(\mu)$ must contain a non-trivial singular vector of degree $2$ and weight  $\nu  = (n-1) \omega_1 +  \omega_2 + \omega_3$  (for $\mu = n \omega_1$) or  $\nu  = (n-1) \omega_3 + \omega_1+  \omega_2$ (for $\mu = n \omega_3$). The proof follows.

\end{proof}

The explicit formulas for these singular vectors of degree two are very complicated and they   will be given in Proposition \ref{prop-formula-singv-mod}.
Now for $\mu = n\omega_1$ and $\nu =  (n-1) \omega_1 +  \omega_2 + \omega_3$ (resp.
for $\mu = n\omega_3 $  and  $\nu  = (n-1) \omega_3 + \omega_1+  \omega_2$), we can define the quotient module:

\begin{equation} \label{rel-M-ni-1}
M(\mu) = \frac{V^k(\mu)}{ V^k(\g). w_{\nu}}.
\end{equation}
Note that in general we don't claim that $M(\mu)$ is an $\widetilde L_k(\g)$--module, but any $\widetilde L_k(\g)$--module in $KL_k$ must be a quotient of $M(\mu)$.
Since the degree of a singular vector must be a positive integer, we  get the following description of the universal $\widetilde L_k(\g)$--modules in $KL_k$: 

\begin{proposition} For the universal  $\widetilde L_{-5/2}(\g)$--module $\overline{M}(n \omega_i)$ in $KL_{-5/2}$ with $\g$--weight $n \omega_i$, for $i=1,3$ we have:
\begin{itemize}
	\item $\overline{M}(n \omega_i) = M(n \omega_i)$ if $\frac{2}{3}n \notin {\Bbb Z}_{\ge 0}$,
	\item  $\overline{M}(n \omega_i)$ is a quotient of    $M(n \omega_i)$ by a singular vector $w_{\nu'}$ of weight $\nu'= n \omega_i + 2 \omega_2$, if  $\frac{2}{3}n \in {\Bbb Z}_{\ge 0}$ and if such singular vector exists.
\end{itemize}   
\end{proposition}

\subsection{Formulas for singular vectors of degree $2$ in $V^k (n \omega_1)$ and $V^k(n \omega_3)$  and their consequences}

 \label{subsect-singv-modules}

In this subsection we determine the explicit formulas for singular vectors of degree two in $V^k(\mu)$, and their images in $H_{f_{subreg}} (V^k(\mu ))$, for $\mu = n \omega_1$ and $\mu = n \omega_3$. We omit details of the proofs because of their similarities with some proofs from previous sections.

Let $\g =sl(4)$ and $k =-\frac{5}{2}$.

\begin{proposition} \label{prop-formula-singv-mod}
Let $n \in {\Bbb Z}_{>0}$. We have:
\begin{itemize}	
\item[(1)] For  $\mu = n \omega_1$, the following vector  $w_{\nu}$ is the (unique, up to a scalar) singular vector in $V^{k}(\mu)$ of degree~$2$ and $\mathfrak{g}$--weight $\nu=(n-1)\omega_1 + \omega_2 + \omega_3$:
\begin{align*}
	w_{\nu}&= 3n ( \frac{3}{2} +n )e_{\varepsilon_2 - \varepsilon_3}(-1)e_{\varepsilon_3 - \varepsilon_4}(-1)v_{\mu}     -n(\frac{19}{4} + \frac{3}{2}n) e_{\varepsilon_2 - \varepsilon_4}(-2)v_{\mu} \\
	&+ n ( \frac{3}{2} +n )h_{\varepsilon_2 - \varepsilon_3}(-1)e_{\varepsilon_2 - \varepsilon_4}(-1)v_{\mu}  - n ( \frac{3}{2} +n )h_{\varepsilon_3 - \varepsilon_4}(-1)e_{\varepsilon_2 - \varepsilon_4}(-1)v_{\mu}\\
	&+(\frac{19}{4} + \frac{3}{2}n) e_{\varepsilon_1 - \varepsilon_4}(-2)f_{\varepsilon_1 - \varepsilon_2}(0)v_{\mu}   -\frac{5}{2} h_{\varepsilon_1 - \varepsilon_2}(-1)e_{\varepsilon_1 - \varepsilon_4}(-1)f_{\varepsilon_1 - \varepsilon_2}(0)v_{\mu} \\
	&-(\frac{3}{2} + n) h_{\varepsilon_2 - \varepsilon_3}(-1)e_{\varepsilon_1 - \varepsilon_4}(-1)f_{\varepsilon_1 - \varepsilon_2}(0)v_{\mu}  + (\frac{3}{2} + n) h_{\varepsilon_3 - \varepsilon_4}(-1)e_{\varepsilon_1 - \varepsilon_4}(-1)f_{\varepsilon_1 - \varepsilon_2}(0)v_{\mu} \\
	&-(\frac{3}{2} + n) e_{\varepsilon_1 - \varepsilon_2}(-1)e_{\varepsilon_2 - \varepsilon_4}(-1)f_{\varepsilon_1 - \varepsilon_2}(0)v_{\mu} -3(\frac{3}{2} + n)e_{\varepsilon_1 - \varepsilon_3}(-1)e_{\varepsilon_3 - \varepsilon_4}(-1)f_{\varepsilon_1 - \varepsilon_2}(0)v_{\mu} \\
	&+ 2(\frac{3}{2} + n)e_{\varepsilon_1 - \varepsilon_3}(-1)e_{\varepsilon_2 - \varepsilon_4}(-1)f_{\varepsilon_1 - \varepsilon_3}(0)v_{\mu}  -(2+3n)e_{\varepsilon_1 - \varepsilon_4}(-1)e_{\varepsilon_2 - \varepsilon_3}(-1)f_{\varepsilon_1 - \varepsilon_3}(0) v_{\mu}\\
	&+(1-n)e_{\varepsilon_1 - \varepsilon_4}(-1)e_{\varepsilon_2 - \varepsilon_4}(-1)f_{\varepsilon_1 - \varepsilon_4}(0)v_{\mu} +e_{\varepsilon_1 - \varepsilon_2}(-1)e_{\varepsilon_1 - \varepsilon_4}(-1)f_{\varepsilon_1 - \varepsilon_2}(0)^2 v_{\mu}\\
	&+ e_{\varepsilon_1 - \varepsilon_3}(-1)e_{\varepsilon_1 - \varepsilon_4}(-1)f_{\varepsilon_1 - \varepsilon_2}(0)f_{\varepsilon_1 - \varepsilon_3}(0)v_{\mu} + e_{\varepsilon_1 - \varepsilon_4}(-1)^2f_{\varepsilon_1 - \varepsilon_2}(0)f_{\varepsilon_1 - \varepsilon_4}(0) v_{\mu}\\
	&+\frac{5}{2}n f_{\varepsilon_1 - \varepsilon_2}(-1)e_{\varepsilon_1 - \varepsilon_4}(-1)v_{\mu}.
\end{align*}
\item[(2)] For  $\mu = n \omega_3$, the following vector  $w_{\nu}$ is the (unique, up to a scalar) singular vector in $V^{k}(\mu)$ of degree~$2$ and $\mathfrak{g}$--weight $\nu=(n-1)\omega_3 + \omega_2 + \omega_1$:
\begin{align*}
	w_{\nu}&= 3n ( \frac{3}{2} +n )e_{\varepsilon_2 - \varepsilon_3}(-1)e_{\varepsilon_1 - \varepsilon_2}(-1)v_{\mu}     +n(\frac{19}{4} + \frac{3}{2}n) e_{\varepsilon_1 - \varepsilon_3}(-2)v_{\mu} \\
	&- n ( \frac{3}{2} +n )h_{\varepsilon_2 - \varepsilon_3}(-1)e_{\varepsilon_1 - \varepsilon_3}(-1)v_{\mu}  + n ( \frac{3}{2} +n )h_{\varepsilon_1 - \varepsilon_2}(-1)e_{\varepsilon_1 - \varepsilon_3}(-1)v_{\mu}\\
	&+(\frac{19}{4} + \frac{3}{2}n) e_{\varepsilon_1 - \varepsilon_4}(-2)f_{\varepsilon_3 - \varepsilon_4}(0)v_{\mu}   -\frac{5}{2} h_{\varepsilon_3 - \varepsilon_4}(-1)e_{\varepsilon_1 - \varepsilon_4}(-1)f_{\varepsilon_3 - \varepsilon_4}(0)v_{\mu} \\
	&-(\frac{3}{2} + n) h_{\varepsilon_2 - \varepsilon_3}(-1)e_{\varepsilon_1 - \varepsilon_4}(-1)f_{\varepsilon_3 - \varepsilon_4}(0)v_{\mu}  + (\frac{3}{2} + n) h_{\varepsilon_1 - \varepsilon_2}(-1)e_{\varepsilon_1 - \varepsilon_4}(-1)f_{\varepsilon_3 - \varepsilon_4}(0)v_{\mu} \\
	&+(\frac{3}{2} + n) e_{\varepsilon_3 - \varepsilon_4}(-1)e_{\varepsilon_1 - \varepsilon_3}(-1)f_{\varepsilon_3 - \varepsilon_4}(0)v_{\mu} +3(\frac{3}{2} + n)e_{\varepsilon_2 - \varepsilon_4}(-1)e_{\varepsilon_1 - \varepsilon_2}(-1)f_{\varepsilon_3 - \varepsilon_4}(0)v_{\mu} \\
	&- 2(\frac{3}{2} + n)e_{\varepsilon_2 - \varepsilon_4}(-1)e_{\varepsilon_1 - \varepsilon_3}(-1)f_{\varepsilon_2 - \varepsilon_4}(0)v_{\mu}  +(2+3n)e_{\varepsilon_1 - \varepsilon_4}(-1)e_{\varepsilon_2 - \varepsilon_3}(-1)f_{\varepsilon_2 - \varepsilon_4}(0) v_{\mu}\\
	&-(1-n)e_{\varepsilon_1 - \varepsilon_4}(-1)e_{\varepsilon_1 - \varepsilon_3}(-1)f_{\varepsilon_1 - \varepsilon_4}(0)v_{\mu} +e_{\varepsilon_3 - \varepsilon_4}(-1)e_{\varepsilon_1 - \varepsilon_4}(-1)f_{\varepsilon_3 - \varepsilon_4}(0)^2 v_{\mu}\\
	&+ e_{\varepsilon_2 - \varepsilon_4}(-1)e_{\varepsilon_1 - \varepsilon_4}(-1)f_{\varepsilon_3 - \varepsilon_4}(0)f_{\varepsilon_2 - \varepsilon_4}(0)v_{\mu} + e_{\varepsilon_1 - \varepsilon_4}(-1)^2f_{\varepsilon_3 - \varepsilon_4}(0)f_{\varepsilon_1 - \varepsilon_4}(0) v_{\mu}\\
	&+\frac{5}{2}n f_{\varepsilon_3 - \varepsilon_4}(-1)e_{\varepsilon_1 - \varepsilon_4}(-1)v_{\mu}.
\end{align*}
\end{itemize}
\end{proposition}
\begin{proof}
Claim (1) is proved using direct verification of relations $e_{\varepsilon_i-\varepsilon_{i+1}}(0).w_{\nu}=0$ for $i=1,2,3$ and $f_\theta(1).w_{\nu}=0$.  Claim (2) now follows from claim (1) using the automorphism $\sigma$ from Remark \ref{automorphism}.
\end{proof}

In the next proposition we determine the images of vectors $w_{\nu}$ from Proposition 
\ref{prop-formula-singv-mod} in $R_{V^{k}(\mu)} = V^{k}(\mu) / C_2(V^{k}(\mu))$, for $\mu = n \omega_1, n \omega_3$. For $x \in \mathfrak{g}$ and $u \in V(\mu)$, denote by $\{x, u\} $ the action of $x$ on $u$. We use the notation $v_{\mu}$ for the image of $v_{\mu} \in V^{k}(\mu)$ in $R_{V^{k}(\mu)}$.

\begin{proposition}  
	Let $n \in {\Bbb Z}_{>0}$. We have:
	\begin{itemize}	
		\item[(1)] For  $\mu = n \omega_1$, the image of singular vector $w_{\nu}$ from Proposition \ref{prop-formula-singv-mod} (1) in $R_{V^{k}(\mu)}$ is equal to:
		\begin{align*}
			w_{\nu}''&= 3n ( \frac{3}{2} +n )e_{\varepsilon_2 - \varepsilon_3}e_{\varepsilon_3 - \varepsilon_4}v_{\mu}    
		+ n ( \frac{3}{2} +n )h_{\varepsilon_2 - \varepsilon_3}e_{\varepsilon_2 - \varepsilon_4}v_{\mu}  - n ( \frac{3}{2} +n )h_{\varepsilon_3 - \varepsilon_4}e_{\varepsilon_2 - \varepsilon_4}v_{\mu}\\
			&  -\frac{5}{2} h_{\varepsilon_1 - \varepsilon_2}e_{\varepsilon_1 - \varepsilon_4} \{ f_{\varepsilon_1 - \varepsilon_2}, v_{\mu} \} -(\frac{3}{2} + n) h_{\varepsilon_2 - \varepsilon_3}e_{\varepsilon_1 - \varepsilon_4} \{ f_{\varepsilon_1 - \varepsilon_2} , v_{\mu} \}  + (\frac{3}{2} + n) h_{\varepsilon_3 - \varepsilon_4}e_{\varepsilon_1 - \varepsilon_4} \{ f_{\varepsilon_1 - \varepsilon_2} , v_{\mu} \} \\
			&-(\frac{3}{2} + n) e_{\varepsilon_1 - \varepsilon_2}e_{\varepsilon_2 - \varepsilon_4} \{ f_{\varepsilon_1 - \varepsilon_2} , v_{\mu} \} -3(\frac{3}{2} + n)e_{\varepsilon_1 - \varepsilon_3}e_{\varepsilon_3 - \varepsilon_4} \{ f_{\varepsilon_1 - \varepsilon_2} , v_{\mu} \} \\
			&+ 2(\frac{3}{2} + n)e_{\varepsilon_1 - \varepsilon_3}e_{\varepsilon_2 - \varepsilon_4} \{ f_{\varepsilon_1 - \varepsilon_3} , v_{\mu} \} -(2+3n)e_{\varepsilon_1 - \varepsilon_4}e_{\varepsilon_2 - \varepsilon_3} \{ f_{\varepsilon_1 - \varepsilon_3} , v_{\mu} \} \\
			&+(1-n)e_{\varepsilon_1 - \varepsilon_4} e_{\varepsilon_2 - \varepsilon_4} \{ f_{\varepsilon_1 - \varepsilon_4} , v_{\mu} \} +e_{\varepsilon_1 - \varepsilon_2}e_{\varepsilon_1 - \varepsilon_4} \{ f_{\varepsilon_1 - \varepsilon_2} , \{ f_{\varepsilon_1 - \varepsilon_2} , v_{\mu} \} \}\\
			&+ e_{\varepsilon_1 - \varepsilon_3}e_{\varepsilon_1 - \varepsilon_4} \{ f_{\varepsilon_1 - \varepsilon_2}, \{ f_{\varepsilon_1 - \varepsilon_3} , v_{\mu} \} \} + e_{\varepsilon_1 - \varepsilon_4}^2 \{ f_{\varepsilon_1 - \varepsilon_2} , \{ f_{\varepsilon_1 - \varepsilon_4} , v_{\mu} \} \} +\frac{5}{2}n f_{\varepsilon_1 - \varepsilon_2}e_{\varepsilon_1 - \varepsilon_4}v_{\mu}.
		\end{align*}
		\item[(2)] For  $\mu = n \omega_3$, the image of singular vector $w_{\nu}$ from Proposition \ref{prop-formula-singv-mod} (2) in $R_{V^{k}(\mu)}$ is equal to:
		\begin{align*}
			w_{\nu}''&= 3n ( \frac{3}{2} +n )e_{\varepsilon_2 - \varepsilon_3}e_{\varepsilon_1 - \varepsilon_2}v_{\mu}    
			- n ( \frac{3}{2} +n )h_{\varepsilon_2 - \varepsilon_3}e_{\varepsilon_1 - \varepsilon_3}v_{\mu}  + n ( \frac{3}{2} +n )h_{\varepsilon_1 - \varepsilon_2}e_{\varepsilon_1 - \varepsilon_3}v_{\mu}\\
			&  -\frac{5}{2} h_{\varepsilon_3 - \varepsilon_4}e_{\varepsilon_1 - \varepsilon_4} \{ f_{\varepsilon_3 - \varepsilon_4} , v_{\mu} \} 
			-(\frac{3}{2} + n) h_{\varepsilon_2 - \varepsilon_3}e_{\varepsilon_1 - \varepsilon_4} \{ f_{\varepsilon_3 - \varepsilon_4} , v_{\mu} \}  + (\frac{3}{2} + n) h_{\varepsilon_1 - \varepsilon_2}e_{\varepsilon_1 - \varepsilon_4} \{ f_{\varepsilon_3 - \varepsilon_4} , v_{\mu} \} \\
			&+(\frac{3}{2} + n) e_{\varepsilon_3 - \varepsilon_4}e_{\varepsilon_1 - \varepsilon_3} \{ f_{\varepsilon_3 - \varepsilon_4} , v_{\mu} \} +3(\frac{3}{2} + n)e_{\varepsilon_2 - \varepsilon_4}e_{\varepsilon_1 - \varepsilon_2} \{ f_{\varepsilon_3 - \varepsilon_4} , v_{\mu} \} \\
			&- 2(\frac{3}{2} + n)e_{\varepsilon_2 - \varepsilon_4}e_{\varepsilon_1 - \varepsilon_3} \{ f_{\varepsilon_2 - \varepsilon_4} , v_{\mu} \} +(2+3n)e_{\varepsilon_1 - \varepsilon_4}e_{\varepsilon_2 - \varepsilon_3} \{ f_{\varepsilon_2 - \varepsilon_4} , v_{\mu} \} \\
			&-(1-n)e_{\varepsilon_1 - \varepsilon_4}e_{\varepsilon_1 - \varepsilon_3} \{ f_{\varepsilon_1 - \varepsilon_4} , v_{\mu} \} +e_{\varepsilon_3 - \varepsilon_4}e_{\varepsilon_1 - \varepsilon_4} \{ f_{\varepsilon_3 - \varepsilon_4} , \{ f_{\varepsilon_3 - \varepsilon_4} , v_{\mu} \} \} \\
			&+ e_{\varepsilon_2 - \varepsilon_4}e_{\varepsilon_1 - \varepsilon_4} \{ f_{\varepsilon_3 - \varepsilon_4} , \{  f_{\varepsilon_2 - \varepsilon_4} , v_{\mu} \} \} + e_{\varepsilon_1 - \varepsilon_4}^2 \{ f_{\varepsilon_3 - \varepsilon_4} , \{ f_{\varepsilon_1 - \varepsilon_4} , v_{\mu} \} \}
			+\frac{5}{2}n f_{\varepsilon_3 - \varepsilon_4}e_{\varepsilon_1 - \varepsilon_4}v_{\mu}.
		\end{align*}
	\end{itemize}
\end{proposition}
\begin{proof}
	Straightforward.
\end{proof}

The proof of the following Lemma is similar to the proof of Lemma \ref{lem-singv-modJ}:

\begin{lemma}  \label{lem-mapping}
	Let $n \in {\Bbb Z}_{>0}$. Let $f_{subreg}$ be the subregular nilpotent element defined by relation 	(\ref{rel-f-subreg}), $x$ a semisimple element defined by (\ref{rel-x-grading}), 	and $J_\chi$ the associated ideal in $\mathcal{S}(\mathfrak{g})$ defined by relation (\ref{rel-ideal-J}). We have:
\begin{itemize}	
	\item[(1)] For  $\mu = n \omega_1$,
	$$w_\nu'' \equiv 3n(\frac{3}{2}+n)v_\mu \ (\textrm{mod} \ J_\chi R_{V^{k}(\mu)}). $$
	\item[(2)] For  $\mu = n \omega_3$,
	$$ w_\nu'' \equiv 3n(\frac{3}{2}+n )e_{\varepsilon_1 - \varepsilon_2} v_\mu \ (\textrm{mod} \ J_\chi R_{V^{k}(\mu)}). $$
\end{itemize}	
\end{lemma}

\begin{remark} \label{general-comments-vanishing}
	We shall now recall some properties of the action of $H_{f_{subreg}}$ on universal $V^k(\g)$--modules. 
	V. Kac and M. Wakimoto proved in \cite[Section 6]{KW3} a general result on the  existence of Verma module for 
	$ W^k(\g, f)$ obtained using QHR from the Verma module for $V^k(\g)$. They calculated the characters of Verma modules and showed that they have a PBW basis. Their approach can be modified   for generalized Verma modules (see also \cite{Ara-11}). Applying this in our setting ($\g =sl(4)$ and $k =-\frac{5}{2}$), we conclude that
	$$  H_{f_{subreg}} (V^k(\mu )) \ne \{ 0\}, \  \mbox{for} \  \mu \in \{ n \omega_1, n \omega_3 \}.$$	
	
Furthermore, $H_{f_{subreg}} (V^k(\mu ))$ is a $W^k(\g, f_{subreg})$--module, with highest weight vector $v_{\mu}^{W}$ having $(\bar L(0), J(0))$--weight determined by:
\begin{eqnarray*}
	&& \bar L(0) v_{\mu}^{W} = \left( \frac{ (\mu \, | \, \mu + 2 \rho)}{2 (k+ h^{\vee})} - \mu (x)  \right) v_{\mu}^{W}, \\
	&& J(0) v_{\mu}^{W} = (\mu \, | \,  \omega _1) \, v_{\mu}^{W}.
\end{eqnarray*}	
\end{remark}

As in Remark \ref{general-comments-vanishing}, denote by $v_{\mu}^{W}$ the highest weight vector of $H_{f_{subreg}} (V^k(\mu ))$, for $\mu = n \omega_1, n \omega_3$. The proof of the following Lemma is similar to the proof of Lemma \ref{lem-img-sing}:

\begin{lemma} \label{lem-img-sing-modules}
	 Let $n \in {\Bbb Z}_{>0}$. We have: 
	\begin{itemize}	
		\item[(1)] For  $\mu = n \omega_1$, the image of singular vector 
		$w_{\nu}$ from Proposition \ref{prop-formula-singv-mod} (1)
		in $H_{f_{subreg}} (V^k(\mu ))$ coincides (up to a non-zero scalar) with the vector $v_{\mu}^{W}$.
		\item[(2)] For  $\mu = n \omega_3$, the image of singular vector 
		$w_{\nu}$ from Proposition \ref{prop-formula-singv-mod} (2)
		in $H_{f_{subreg}} (V^k(\mu ))$ coincides (up to a non-zero scalar) with the vector $G^{+} (-1) v_{\mu}^{W}$.		
	\end{itemize}	
\end{lemma}

\subsection{Proof of Theorem \ref{conj-p}}

Proofs of the following two results use explicit expressions for singular vectors $w_{\nu}$ from  Proposition \ref{prop-formula-singv-mod}, and their consequences discussed in   Subsection  \ref{subsect-singv-modules}.
	 
\begin{proposition}  \label{prop-zero}  For $n \in {\Bbb Z}_{>0}$ we have:
	\begin{itemize}
		\item $H_{f_{subreg}}(\overline{M}(n \omega_1))  = \{0 \}$. 
		\item   $H_{f_{subreg}}(M)  = \{0 \}$, for any highest weight $\widetilde L_{-5/2} (\g)$--module $M$ in $KL_{-5/2}$ with $\g$--weight $n \omega_1$.
	\end{itemize}
\end{proposition}

\begin{proof} 
  Lemma \ref{lem-img-sing-modules} gives  that the image of the singular vector $w_{\nu}$ in  the $W^k(\g, f_{subreg})$--module \\ $H_{f_{subreg}} (V^k(n \omega_1))$ coincides with its  highest weight vector. This implies that $H_{f_{subreg}}({M}(n \omega_1)) =\{0\}$ and therefore $H_{f_{subreg}}(\overline{M}(n \omega_1))  = \{0 \}$. 
   The proof of the second assertion follows from the first assertion and  the fact that  any highest weight $\widetilde L_k (\g)$--module in $KL_{-5/2}$ of $\g$--weight 
	$n \omega_1$ must be a quotient of  $\overline{M}(n \omega_1)$.
	
\end{proof}

\begin{proposition} \label{prop-non-zero}
	For any highest weight $\widetilde L_{-5/2}(\g)$--module $M$  in $KL_{-5/2}$ with $\g$--weight $n \omega_3$, we have $H_{f_{subreg}}(M)  = M_J (1, \frac{1}{4} n) \ne \{0\}$.
\end{proposition}
\begin{proof}
	Lemma  \ref{lem-img-sing-modules}
implies that the  singular vector $w_{\nu}$  is mapped by $H_{f_{subreg}}( \, \cdot \,)$ to a singular vector of non-zero degree in  $H_{f_{subreg}}(V^k(n \omega_3))$. The calculation of conformal weights shows that 
	 $w_{\nu '}$ is  mapped by $H_{f_{subreg}}( \, \cdot \,)$ to a singular vector of non-zero degree. Therefore the image of generators of $J^k \cdot V^k(n \omega_3)$  can not coincide with highest weight vector of $H_{f_{subreg}}(V^{k} (n \omega_3))$. This proves that  $H_{f_{subreg}}(\overline{M}(n \omega_3))  \ne \{0\}$. Since $H_{f_{subreg}}(\widetilde L_{-5/2} (\g)) \cong M_J(1)$, we conclude that $H_{f_{subreg}}(\overline{M}(n \omega_3)) \cong M_J (1, \frac{1}{4} n)$.
	
	Take any quotient of $\overline{M}(n \omega_3)$. Then it can have singular vectors of weights $m \omega_1$ or $m \omega_3$. But  submodules generated by singular vectors of weights $m \omega_1$ are mapped to zero by Proposition \ref{prop-zero}.
	
	On the other hand, singular vectors of weight $m \omega_3$ have $\bar L(0)$--weight
	$ \frac{m(m-1)}{4} \ne  \frac{n(n-1)}{4}$ for all $m \in {\Bbb Z}_{\ge 0}$, $m \ne n$. This implies that we can not find any singular vector in $\overline{M}(n \omega_3)$ which is mapped to a highest weight vector of $M_J (1, \frac{1}{4} n)$.  The proof follows.
\end{proof}

\begin{remark}
	Propositions \ref{prop-zero} and \ref{prop-non-zero} prove Theorem \ref{conj-p}.
\end{remark}

\section{Conformal embedding $gl(4)  \hookrightarrow sl(5)$ at $k=-5/2$ and the fusion rules for irreducible $L_{-5/2}( sl(4) )$--modules in the category $KL_{-5/2}$}

In this section we recall the results on conformal embedding $gl(4)  \hookrightarrow sl(5)$ at $k=-5/2$ from \cite{AKMPP-16}. We use this conformal embedding to determine the fusion rules for irreducible $L_{-5/2}(sl(4))$--modules in the category $KL_{-5/2}$. As a consequence, we get that $KL_{-5/2}$ is a rigid braided tensor category.

   Let $V = L_{-5/2} (sl(5))$. Take $c$ from the Cartan subalgebra of $\mathfrak g = sl(5)$ such that
$$ \mathfrak g = \mathfrak g_{-1} \oplus \mathfrak g_0 \oplus \mathfrak g_{1},$$
$$\mathfrak g_0 = gl(4) = sl(4) + {\Bbb C} c,$$
$$ \mathfrak g_{1} =V_{sl(4)} (\omega_1), \quad \mathfrak g_{-1} = V_{sl(4)} (\omega_3)$$
$$ c \equiv j \ \mbox{id} \quad \mbox{on} \  \mathfrak g_{j},   \ j \in \{-1,0,1\}. $$
Let $M_c(1)$ be the Heisenberg subalgebra of $V$ generated by $c$. Let $M_c(1, s)$ denote irreducible $M_c(1)$--module on which $c(0)$ acts as $s \ \mbox{id}$.

\begin{proposition} \label{konf-embedding} \cite{AKMPP-16}
	There is a conformal embedding of $L_{-5/2} (sl(4)) \otimes M_{c}  (1)$ into $V$ such that
	$$  V = \bigoplus_{ s \in {\Bbb Z} } V ^{(s)} \quad c(0)  \equiv s \  \mbox{id} \ \mbox{on}  \  V ^{(s)}, $$
	and each $V^{(s)}$ is an irreducible  $L_{-5/2} (sl(4)) \otimes M_{c}  (1)$--module.
\end{proposition}

\begin{proposition} \label{singular}  Assume that $ \mathcal M$ is an irreducible $V$--module such that $c(0)$ acts semisimply on $\mathcal M$:
	$$  \mathcal M = \bigoplus_{ s \in \Delta + {\Bbb Z} }  \mathcal M ^{(s)} \quad c(0)  \equiv s \  \mbox{id} \ \mbox{on}  \  \mathcal M ^{(s)}.$$
	Then we have:
	\begin{itemize}
		\item[(1)]  Each   $\mathcal M ^{(s)}$ is an irreducible  $L_{-5/2} (sl(4)) \otimes M_{c}  (1)$--module.
		\item[(2)]  If there is a vector $v \in \mathcal M$ which is a highest weight vector for $\widehat{sl(4)}$ of weight $\lambda$ such that $c(0) v = s v$ for certain $s$, then 
		$ \mathcal{U}(\widehat{sl(4)}). v$ is an irreducible $L_{-5/2} (sl(4))$--module such that
		$$ \mathcal M^{(s)} = L_{-5/2} (\lambda) \otimes M_c (1, s). $$
	\end{itemize}
\end{proposition}
\begin{proof}
	The proof of assertion (1) uses Proposition \ref{konf-embedding} and  completely analogous  arguments to those of \cite[Theorem 5.1]{AdP-JMP}.
	
	Assertion (2) follows easily from (1). 
\end{proof}

	We have the following decomposition of $L_{-5/2} (sl(5))$ as an $L_{-5/2} (gl(4)) = L_{-5/2} (sl(4)) \otimes M_c(1)$--module (cf. \cite{AKMPP-16}):
	\bea L_{-5/2} (sl(5)) = \bigoplus_{ n= 0} ^{\infty} L_{-5/2} ( n  \omega_1 ) \otimes M_c(1,n)  \   \oplus   \  \bigoplus_{ n= 1} ^{\infty} L_{-5/2}  ( n \omega_3) \otimes M_c(1,-n) . \label{dec-rig} \eea

We introduce the following notation for irreducible $L_{-5/2}(sl(4))$--modules in the category $KL_{-5/2}$: $$\pi_i=L_{-5/2}(  i \omega_1), \ \pi_{-i}=L_{-5/2}(  i \omega_3), \ i \in \mathbb{Z}_{\geq 0}. $$

The top component is $$\pi_i(0) = V(i \omega_1), \ \pi_{-i} (0)  = V(i \omega_3).$$

The proof of the following proposition is completely analogous to the proof of \cite[Theorem 6.2]{AP-JAA}.

\begin{proposition} \label{fusion-rules}
Let $i,j \in \mathbb{Z}$. We have the following fusion rule: $$ \pi_i \times \pi_j = \pi_{i+j} .$$
This means that for $i, j, k \in {\Bbb Z}$:
 $$\dim  I { \pi_{k} \choose \pi_i \ \ \pi_j } = \delta_{i+j,k}. $$ 
\end{proposition}
\begin{proof}
Assume that  $ I { \pi_{k} \choose \pi_i \ \ \pi_j }  \ne \{ 0\}$, i.e, there is a non-trivial intertwining operator $I$ of the type ${ \pi_{k} \choose \pi_i \ \ \pi_j }$. As in \cite{AP-JAA}, this implies that $\pi_{k} (0)$ is a non-trivial summand in the tensor product  of $\g$--modules $\pi_i(0) \otimes \pi_j(0)$. Using the decomposition of tensor product, we see that $\pi_k(0)$ appears in the tensor product
 $\pi_i(0) \otimes \pi_j(0)$ if and only if $k = i+j$, and that the multiplicity is exactly one. So
 $$\dim  I { \pi_{k} \choose \pi_i \ \ \pi_j } \le \delta_{i+j,k}. $$ 
 It remains to construct an non-trivial intertwining operator of type ${ \pi_{i+j} \choose \pi_i \ \ \pi_j } $.
Similarly to the case considered in  \cite{AP-JAA}, by using vertex operator on $V= L_{-5/2} (sl(5))$ we get non-trivial intertwining operator  $\mathcal Y$   from the space  $ I { V^{(i+j)} \choose V ^{(i)} \ \  V^{(j)} }$.
 Since $V^{(s)} = \pi_s \otimes M_c(1, s )$, it follows that $\mathcal Y = \mathcal Y_1 \otimes \mathcal Y_2$, where
 $$   \mathcal Y_1 \in  I { \pi_{i+j} \choose \pi_i \ \ \pi_j }, \    \mathcal Y_2 \in  I { M_c(1, i+j ) \choose M_c (1,  i )  \ \  M_c(1, j ) }.$$
So $\dim  I { \pi_{i+j} \choose \pi_i \ \ \pi_j } =1. $
The proof follows.

\end{proof}

\begin{corollary} $KL_{-5/2}$ is a semi-simple rigid braided tensor category with the fusion rules
$$ \pi_i \boxtimes \pi_j = \pi_{i+j} \quad (i,j \in {\Bbb Z}).$$
\end{corollary}
\begin{proof}
\cite[Theorem 3.3]{CY} implies  that in general,  $KL_k$ is a braided tensor category if the following condition holds:
\begin{itemize}
\item  Every  highest weight $L_k(\g)$--module in $KL_k$ is of finite length.
\end{itemize}
This condition  is satisfied in our case since $KL_k$ is semi-simple. Therefore we have that  $KL_{-5/2}$ is a   braided tensor category.

Note also that from the decomposition (\ref{dec-rig}) it follows that the compact Lie group $U(1)$ acts on  $L_{-5/2} (sl(5))$:
\begin{itemize}
\item $L_{-5/2}(sl(4)) \otimes M_c(1) = L_{-5/2} (sl(5)) ^{U(1)}$,
\item  All  $L_{-5/2}(sl(4))$--modules in  $KL_{-5/2}$ appear in the decomposition of $L_{-5/2}(sl(5))$ as an \linebreak $L_{-5/2} (sl(5)) ^{U(1)} \times U(1)$--module.
\end{itemize} 
Then exactly the same methods used in the  proof of rigidity of examples studied in   \cite[Section 5]{CY} or \cite{ACGY}  work also in the case of $L_{-5/2}(sl(4))$. This proves the rigidity. As a consequence, our fusion rules from Proposition  \ref{fusion-rules} give the fusion rules in the vertex tensor category settings. 
(The details are exactly the same as in the case of the tensor category $KL_{-1}$.)
\end{proof}

\begin{remark}
 
In \cite{CY}, the authors propose a tensor category approach for studying the fusion rules for affine vertex algebras. Assuming that on the category of $V_{k}(\g_0)$--modules in $KL_k$ there exists the structure of a braided tensor category, their result,  together with the decomposition from \cite{AKMPP-16}, should  imply  the fusion rules in $KL_k$. 

Similar phenomena happens for all conformal embeddings $\g_0 \hookrightarrow \g$ at conformal levels $k$  whose decomposition is as in \cite[Theorem 5.1]{AKMPP-16}. We conjectured that all modules for  $V_{k}(\g_0)$  appearing in the decompositions for these conformal embeddings are simple-current modules, which would imply that modules in $KL_k$ have the isomorphic fusion algebras as those of $KL_{-5/2}$.
\end{remark}

\section{Decompositions of irreducible modules in the category $\mathcal{O}$ for the conformal embedding $gl(4)  \hookrightarrow sl(5)$ at $k=-5/2$}
\label{sect-class-simple}

In this section we determine the decompositions of irreducible  $L_{-5/2}( sl(5) )$--modules in the category $\mathcal{O}$ as $L_{-5/2} (sl(4)) \otimes M_c(1)$--modules.  

We recall the classification of irreducible $L_{-5/2}( sl(5))$--modules in the category $\mathcal{O}$.

\begin{proposition} \label{Prop1} \cite{Ara, Pe-07}
	The complete list of irreducible $L_{-5/2}(sl(5))$--modules in the category 
	$\mathcal{O}$ is given by the set
	 $$\{  L_{-5/2}(\lambda_i) \ |\ i=1,\ldots,16\}, $$
	where: \\
	\begin{tabular}{ p{5cm} p{8cm} }
		$\lambda_1=0$, & $\lambda_9=-\frac{1}{2}\omega_2-\frac{1}{2}\omega_3$, \vspace*{0.2cm} \\ 
		$\lambda_2=-\frac{5}{2}\omega_1$, & 	$\lambda_{10}=-\frac{3}{2}\omega_2-\frac{1}{2}\omega_4$, \vspace*{0.2cm} \\  
		$\lambda_3=-\frac{5}{2}\omega_2$, & 	$\lambda_{11}=-\frac{3}{2}\omega_3+\frac{1}{2}\omega_4$,  \vspace*{0.2cm}\\ 
		$\lambda_4=-\frac{5}{2}\omega_3$, & 	$\lambda_{12}=-\frac{3}{2}\omega_1+\frac{1}{2}\omega_2-\frac{3}{2}\omega_3$, \vspace*{0.2cm} \\ 
		$\lambda_5=-\frac{5}{2}\omega_4$, & 	$\lambda_{13}=-\frac{1}{2}\omega_1-\frac{1}{2}\omega_2-\frac{3}{2}\omega_4$, \vspace*{0.2cm} \\ 
		$\lambda_6=\frac{1}{2}\omega_1-\frac{3}{2}\omega_2$, & $\lambda_{14}=-\frac{3}{2}\omega_1-\frac{1}{2}\omega_3-\frac{1}{2}\omega_4$,  \vspace*{0.2cm}\\ 
		$\lambda_7=-\frac{1}{2}\omega_1-\frac{3}{2}\omega_3$, & $\lambda_{15}=-\frac{3}{2}\omega_2+\frac{1}{2}\omega_3-\frac{3}{2}\omega_4$,  \vspace*{0.2cm}\\ 
		$\lambda_8=-\frac{3}{2}\omega_1-\frac{3}{2}\omega_4$, & $\lambda_{16}=-\frac{1}{2}\omega_1-\frac{1}{2}\omega_2-\frac{1}{2}\omega_3-\frac{1}{2}\omega_4.$   
	\end{tabular}
\end{proposition}

Let us denote by $v_{\lambda_i}$ the highest weight vector of $L_{-5/2}(sl(5))$--module
$L_{-5/2}(\lambda_i)$, for $i=1,\ldots,16$. In order to decompose irreducible $L_{-5/2}( sl(5) )$--modules from Proposition \ref{Prop1} as $L_{-5/2} (sl(4)) \otimes M_c(1)$--modules, we will determine explicit formulas for $\widehat{gl(4)}$--singular vectors in these modules. First, we determine the singular vectors of lowest conformal weight:

\begin{proposition}\label{sing1}  For each $i =1, \dots, 16$ let $ {\lambda}_i=\sum_{k=1}^{4} a_{i,k} \omega_k.$ If $a_{i,j} \neq 0$ and $a_{i,k}=0$ for $k>j$, then $ f_{\varepsilon_j - \varepsilon_5} (0) ^{n}  v_{\lambda_i}$ is a (non-trivial)  singular vector for  $L_{-5/2} (gl(4)) = L_{-5/2} (sl(4)) \otimes M_c(1).$
\end{proposition}
\begin{proof}
Straightforward calculation.
\end{proof}

The next proposition gives the remaining singular vectors:

\begin{proposition}\label{sing2} 
	\begin{itemize}
		\item [(1)]For any $i =1, \dots, 16$ such that $\widehat {\lambda}_i(\alpha_{0}^{\vee}) \notin {\Bbb Z}_{\ge 0}$, and any $n \in {\Bbb Z}_{\ge 0}$, $  e_{\varepsilon_1 - \varepsilon_5} (-1) ^{n}  v_{\lambda_i}$ is a (non-trivial) singular vector for  $L_{-5/2} (gl(4)) = L_{-5/2} (sl(4)) \otimes M_c(1).$
		\item [(2)]	The following table contains formulas for (non-trivial) singular vectors for $L_{-5/2} (gl(4)) = L_{-5/2} (sl(4)) \otimes M_c(1)$ in $L_{-5/2}(\lambda_i)$, for $i=2,3,4,5,12,13,14,15$, and for any $n \in {\Bbb Z}_{\ge 0}$:
		\begin{center}
			\begin{tabular}{ |c|c|c| } 
				\hline
				$\lambda_i$ & singular vectors  \\ 
				\hline \hline
				$\lambda_2, \lambda_{12}, \lambda_{13}, \lambda_{14}$ & $e_{\varepsilon_2-\varepsilon_5}(-1)^n v_{\lambda_i}$   \\ 
				\hline
				$\lambda_3, \lambda_{15}$ & $e_{\varepsilon_3-\varepsilon_5}(-1)^n v_{\lambda_i}$  \\ 
				\hline
				$\lambda_4$ & $e_{\varepsilon_4-\varepsilon_5}(-1)^n v_{\lambda_i}$  \\ 
				\hline
				$\lambda_5$ & $e_{\varepsilon_1-\varepsilon_5}(-2)^n v_{\lambda_i}$  \\ 
				\hline
			\end{tabular}
		\end{center}
	\end{itemize} 
\end{proposition}
\begin{proof}
Straightforward calculation. Since the condition $\widehat {\lambda}_i(\alpha_{0}^{\vee}) \notin {\Bbb Z}_{\ge 0}$ is satisfied for $i=1,6,7,8,9,10,11,16$, the assertion (1) gives the singular vectors in these cases. The remaining cases are given by assertion (2).
\end{proof}

\begin{theorem} \label{sl5-dekomp}
	    Using the notation from Proposition \ref{Zhu-moduli}, the decompositions of irreducible $L_{-5/2}(sl(5))$--modules in the category $\mathcal{O}$ as
	    $L_{-5/2} (sl(4)) \otimes M_c(1)$--modules are given by the following relations:
		\begin{align*}
		L_{-5/2} (\lambda_1) &= \bigoplus_{ n= 0} ^{\infty} L_{-5/2} ( \mu_1(n) ) \otimes M_c(1,n) \   \oplus   \  \bigoplus_{ n= 1} ^{\infty} L_{-5/2}  ( \mu_2(n) ) \otimes M_c(1,-n) ,\\
		L_{-5/2} (\lambda_2) &= \bigoplus_{ n= 0} ^{\infty} L_{-5/2} ( \mu_1(-n-\tfrac{5}{2}) ) \otimes M_c(1,-\tfrac{1}{2}-n) \   \oplus   \  \bigoplus_{ n= 1} ^{\infty} L_{-5/2}  ( \mu_3(-n-\tfrac{5}{2}) )\otimes M_c(1,-\tfrac{1}{2}+n),\\
		L_{-5/2} (\lambda_3) &= \bigoplus_{ n= 0} ^{\infty} L_{-5/2} ( \mu_3(n) ) \otimes M_c(1,-1-n)  \   \oplus   \  \bigoplus_{ n= 1} ^{\infty} L_{-5/2}  ( \mu_4(-n-\tfrac{5}{2}) ) \otimes M_c(1,-1+n) ,\\
		L_{-5/2} (\lambda_4) &= \bigoplus_{ n= 0} ^{\infty} L_{-5/2} ( \mu_4(n) ) \otimes M_c(1,-\tfrac{3}{2}-n) \   \oplus   \  \bigoplus_{ n= 1} ^{\infty} L_{-5/2}  ( \mu_2(-n-\tfrac{5}{2}) ) \otimes M_c(1,-\tfrac{3}{2}+n),\\
		L_{-5/2} (\lambda_5) &= \bigoplus_{ n= 0} ^{\infty} L_{-5/2} ( \mu_2(n) ) \otimes M_c(1,-2-n) \   \oplus   \  \bigoplus_{ n= 1} ^{\infty} L_{-5/2}  ( \mu_1(n) ) \otimes M_c(1,-2+n), \\
		L_{-5/2} (\lambda_6) &= \bigoplus_{ n= 0} ^{\infty} L_{-5/2} ( \mu_7(n+\tfrac{1}{2})  ) \otimes M_c(1,-\tfrac{1}{2}-n) \   \oplus   \  \bigoplus_{ n= 1} ^{\infty} L_{-5/2}  ( \mu_{ 5}   (n+\tfrac{1}{2})   ) \otimes M_c(1,-\tfrac{1}{2}+n), \\
		L_{-5/2} (\lambda_7) &= \bigoplus_{ n= 0} ^{\infty} L_{-5/2} ( \mu_{14} (n ) )  \otimes M_c(1,-1-n)  \   \oplus   \  \bigoplus_{ n= 1} ^{\infty} L_{-5/2}  ( \mu_{10}  ( n-\tfrac{1}{2}  )  )\otimes M_c(1,-1+n),\\
		L_{-5/2} (\lambda_8) &= \bigoplus_{ n= 0} ^{\infty} L_{-5/2} ( \mu_{9} (n) )  \otimes M_c(1,-\tfrac{3}{2}-n)  \   \oplus   \  \bigoplus_{ n= 1} ^{\infty} L_{-5/2}  ( \mu_{1}  (n-\tfrac{3}{2})  ) \otimes M_c(1,-\tfrac{3}{2}+n), \\
		L_{-5/2} (\lambda_9) &= \bigoplus_{ n= 0} ^{\infty} L_{-5/2} ( \mu_{8} (n-\tfrac{1}{2}) )  \otimes M_c(1,-\tfrac{1}{2}-n)  \   \oplus   \  \bigoplus_{ n= 1} ^{\infty} L_{-5/2}  ( \mu_{15}  (n) ) \otimes M_c(1,-\tfrac{1}{2}+n), \\
		L_{-5/2} (\lambda_{10}) &= \bigoplus_{ n= 0} ^{\infty} L_{-5/2} ( \mu_{6} (n) )  \otimes M_c(1,-1-n)  \   \oplus   \  \bigoplus_{ n= 1} ^{\infty} L_{-5/2}  ( \mu_{5}  (n) ) \otimes M_c(1,-1+n), \\
			L_{-5/2} (\lambda_{11}) &= \bigoplus_{ n= 0} ^{\infty} L_{-5/2} ( \mu_{2} (n-\tfrac{3}{2}) )  \otimes M_c(1,-\tfrac{1}{2}-n)  \   \oplus   \  \bigoplus_{ n= 1} ^{\infty} L_{-5/2}  ( \mu_{10}  (n) ) \otimes M_c(1,-\tfrac{1}{2}+n), \\
		L_{-5/2} (\lambda_{12}) &= \bigoplus_{ n= 0} ^{\infty} L_{-5/2} ( \mu_{11} (n+\tfrac{1}{2}) ) \otimes M_c(1,-1-n)   \   \oplus   \  \bigoplus_{ n= 1} ^{\infty} L_{-5/2}  ( \mu_{12}  (n+\tfrac{1}{2}) ) \otimes M_c(1,-1+n), \\
		L_{-5/2} (\lambda_{13}) &= \bigoplus_{ n= 0} ^{\infty} L_{-5/2} ( \mu_{13} (n) )  \otimes M_c(1,-\tfrac{3}{2}-n)  \   \oplus   \  \bigoplus_{ n= 1} ^{\infty} L_{-5/2}  ( \mu_{7}  (-n-\tfrac{1}{2}) ) \otimes M_c(1,-\tfrac{3}{2}+n), \\
			L_{-5/2} (\lambda_{14}) &= \bigoplus_{ n= 0} ^{\infty} L_{-5/2} ( \mu_{9} (n-\tfrac{1}{2}) )  \otimes M_c(1,-1-n)  \   \oplus   \  \bigoplus_{ n= 1} ^{\infty} L_{-5/2}  ( \mu_{16}  (n) ) \otimes M_c(1,-1+n), \\	
		L_{-5/2} (\lambda_{15}) &= \bigoplus_{ n= 0} ^{\infty} L_{-5/2} ( \mu_{6} (n+\tfrac{1}{2}) ) \otimes M_c(1,-\tfrac{3}{2}-n)    \   \oplus   \  \bigoplus_{ n= 1} ^{\infty} L_{-5/2}  ( \mu_{8}  (-n-\tfrac{3}{2}) ) \otimes M_c(1,-\tfrac{3}{2}+n), \\
		L_{-5/2} (\lambda_{16} ) &= \bigoplus_{ n= 0} ^{\infty} L_{-5/2} ( \mu_{13} (n-\tfrac{1}{2}) ) \otimes M_c(1,-1-n) \   \oplus   \  \bigoplus_{ n= 1} ^{\infty} L_{-5/2}  ( \mu_{15} (n-\tfrac{1}{2}) ) \otimes M_c(1,-1+n). 
		\end{align*}
\end{theorem}

\begin{proof}
	The proof follows from Propositions \ref{singular}, \ref{sing1} and \ref{sing2}.
\end{proof}

\section{$L_{-5/2} (sl(4))$--module $L_{-5/2}( t \omega_1)$ as a subquotient of a relaxed $L_{-5/2} (sl(5))$--module}

In previous sections we proved that $L_{-5/2}(t \omega_1)$ is an irreducible $L_{-5/2}( sl(4))$--module for any $t \in \mathbb{C}$ (see Theorems \ref{kvocijent-moduli} and \ref{main-max}, Remark \ref{rem-klas-simple}). On the other hand, in Theorem \ref{sl5-dekomp} we identified modules $L_{-5/2}( t \omega_1)$ as submodules of $L_{-5/2}(sl(5))$--modules from the category $\mathcal{O}$, only for countably many $t \in \mathbb{C}$. In this section we consider a realization of $L_{-5/2}( t \omega_1)$ as a subquotient of 
an $L_{-5/2} (sl(5))$--module, for any $t \in \mathbb{C}$. It turns out that, in order to obtain such realization, one has to consider modules beyond the category $\mathcal{O}$, more precisely, the relaxed $L_{-5/2} (sl(5))$--modules.
Relaxed modules for affine vertex algebras have also recently been studied in \cite{FMR, Kaw, KR1, KR2}. 

We also prove that there is a non-trivial homomorphism from Zhu's algebra $A(L_{-5/2} (sl(4)))$ to Weyl algebra $\mathcal A_4 $, which implies that any module for  $\mathcal A_4 $ is naturally a module for $A(L_{-5/2} (sl(4)))$, so it can be induced to an $L_{-5/2} (sl(4))$--module (see Proposition \ref{Zhu1} (2)). We hope that this result can be useful for studying $L_{-5/2} (sl(4))$--modules beyond category $\mathcal{O}$.

\subsection{ A homomorphism from Zhu's algebra to Weyl algebra}

Denote by $\mathcal A_4 $ the Weyl algebra with generators $x _1, x _2, x _3, x _4,  \partial  _1, \partial  _2, \partial  _3, \partial  _4$ and non-trivial commutation relations:
$$ [\partial  _i, x _j] = \delta _{ij}, \quad i,j=1,2,3,4.$$
\begin{theorem}
	There is a non-trivial homomorphism $$\Phi \colon A(L_{-5/2} (sl(4))) \rightarrow \mathcal A_4 $$ uniquely determined by
	 \begin{equation} \label{rel-hom-weyl}  \quad e_{\varepsilon_i - \varepsilon_j} \mapsto x _i  \partial  _j,  \quad f_{\varepsilon_i - \varepsilon_j} \mapsto x _j  \partial  _i, \quad 1 \leq i<j \leq 4.
	 \end{equation}	
\end{theorem}

\begin{proof}
	Clearly, there is a non-trivial homomorphism $\Phi \colon \mathcal{U} (sl(4)) \rightarrow \mathcal A_4 $ uniquely determined by relation (\ref{rel-hom-weyl}).
	By direct calculation we obtain 
	\begin{align*}
	&\Phi(e_{\varepsilon_2-\varepsilon_4}e_{\varepsilon_1-\varepsilon_3})=\Phi(e_{\varepsilon_2-\varepsilon_3}e_{\varepsilon_1-\varepsilon_4})\\
	&\Phi(h_1e_{\varepsilon_2-\varepsilon_4}e_{\varepsilon_1-\varepsilon_3})=\Phi(h_{1}e_{\varepsilon_2-\varepsilon_3}e_{\varepsilon_1-\varepsilon_4})\\
	&\Phi(h_{2}e_{\varepsilon_2-\varepsilon_4}e_{\varepsilon_1-\varepsilon_3})=\Phi(h_{2}e_{\varepsilon_2-\varepsilon_3}e_{\varepsilon_1-\varepsilon_4})\\
	&\Phi(h_{3}e_{\varepsilon_2-\varepsilon_4}e_{\varepsilon_1-\varepsilon_3})=\Phi(h_{3}e_{\varepsilon_2-\varepsilon_3}e_{\varepsilon_1-\varepsilon_4})\\
	&\Phi(h_{1}h_{2}e_{\varepsilon_2-\varepsilon_4}e_{\varepsilon_1-\varepsilon_3})=\Phi(h_{1}h_{2}e_{\varepsilon_2-\varepsilon_3}e_{\varepsilon_1-\varepsilon_4})\\
	&\Phi(h_{2}^2e_{\varepsilon_2-\varepsilon_4}e_{\varepsilon_1-\varepsilon_3})=\Phi(h_{2}^2e_{\varepsilon_2-\varepsilon_3}e_{\varepsilon_1-\varepsilon_4})\\
	&\Phi(h_{2}h_{3}e_{\varepsilon_2-\varepsilon_4}e_{\varepsilon_1-\varepsilon_3})=\Phi(h_{2}h_{3}e_{\varepsilon_2-\varepsilon_3}e_{\varepsilon_1-\varepsilon_4})\\
	&\Phi(f_{\varepsilon_1-\varepsilon_2}e_{\varepsilon_1-\varepsilon_2}e_{\varepsilon_2-\varepsilon_4}e_{\varepsilon_1-\varepsilon_3})=\Phi(f_{\varepsilon_1-\varepsilon_2}e_{\varepsilon_1-\varepsilon_2}e_{\varepsilon_2-\varepsilon_3}e_{\varepsilon_1-\varepsilon_4})\\
	&\Phi(f_{\varepsilon_1-\varepsilon_3}e_{\varepsilon_1-\varepsilon_3}e_{\varepsilon_2-\varepsilon_4}e_{\varepsilon_1-\varepsilon_3})=\Phi(f_{\varepsilon_1-\varepsilon_3}e_{\varepsilon_1-\varepsilon_3}e_{\varepsilon_2-\varepsilon_3}e_{\varepsilon_1-\varepsilon_4})\\
	&\Phi(f_{\varepsilon_1-\varepsilon_4}e_{\varepsilon_1-\varepsilon_4}e_{\varepsilon_2-\varepsilon_4}e_{\varepsilon_1-\varepsilon_3})=\Phi(f_{\varepsilon_1-\varepsilon_4}e_{\varepsilon_1-\varepsilon_4}e_{\varepsilon_2-\varepsilon_3}e_{\varepsilon_1-\varepsilon_4})\\
	&\Phi(f_{\varepsilon_2-\varepsilon_3}e_{\varepsilon_2-\varepsilon_3}e_{\varepsilon_2-\varepsilon_4}e_{\varepsilon_1-\varepsilon_3})=\Phi(f_{\varepsilon_2-\varepsilon_3}e_{\varepsilon_2-\varepsilon_3}^2e_{\varepsilon_1-\varepsilon_4})\\
	&\Phi(f_{\varepsilon_2-\varepsilon_4}e_{\varepsilon_2-\varepsilon_4}^2e_{\varepsilon_1-\varepsilon_3})=\Phi(f_{\varepsilon_2-\varepsilon_4}e_{\varepsilon_2-\varepsilon_4}^2e_{\varepsilon_1-\varepsilon_4})\\
	&\Phi(f_{\varepsilon_3-\varepsilon_4}e_{\varepsilon_3-\varepsilon_4}e_{\varepsilon_2-\varepsilon_4}e_{\varepsilon_1-\varepsilon_3})=\Phi(f_{\varepsilon_3-\varepsilon_4}e_{\varepsilon_3-\varepsilon_4}e_{\varepsilon_2-\varepsilon_3}e_{\varepsilon_1-\varepsilon_4})\\
	&\Phi(h_{1}h_{3}e_{\varepsilon_2-\varepsilon_4}e_{\varepsilon_1-\varepsilon_3})=\Phi(h_{1}h_{3}e_{\varepsilon_2-\varepsilon_3}e_{\varepsilon_1-\varepsilon_4})=\\&=x_1^2 x_2 x_3 \partial_1 \partial_3^2 \partial_4 - x_1 x_2^2 x_3 \partial_2 \partial_3^2 \partial_4 - x_1^2 x_2 x_4 \partial_1 \partial_3 \partial_4^2 + x_1 x_2^2 x_4 \partial_2 \partial_3 \partial_4^2\\
	&\Phi(e_{\varepsilon_3-\varepsilon_4}e_{\varepsilon_2-\varepsilon_3}e_{\varepsilon_1-\varepsilon_3})=x_1 x_2 x_3 \partial_3^2 \partial_4\\
	&\Phi(e_{\varepsilon_3-\varepsilon_4}e_{\varepsilon_2-\varepsilon_3}^2 e_{\varepsilon_1-\varepsilon_2})=x_1 x_2^2 x_3 \partial_2 \partial_3^2 \partial_4\\
	&\Phi(h_{3}e_{\varepsilon_1-\varepsilon_2}e_{\varepsilon_2-\varepsilon_3} e_{\varepsilon_2-\varepsilon_4})=2 x_1 x_2 x_3 \partial_3 ^2 \partial_4 + x_1 x_2^2 x_3 \partial_2 \partial_3 ^2 \partial_4 - 2 x_1 x_2 x_4 \partial_3 \partial_4^2 - x_1 x_2^2 x_4 \partial_2 \partial_3 \partial_4^2\\
	&\Phi(h_{1}e_{\varepsilon_3-\varepsilon_4}e_{\varepsilon_2-\varepsilon_3} e_{\varepsilon_1-\varepsilon_3})=x_1^2 x_2 x_3 \partial_1 \partial_3^2 \partial_4 - x_1 x_2^2 x_3 \partial_2 \partial_3 ^2 \partial_4\\
	&\Phi(h_{3}f_{\varepsilon_1-\varepsilon_2}e_{\varepsilon_1-\varepsilon_4} e_{\varepsilon_1-\varepsilon_3})=2 x_1 x_2 x_3 \partial_3 ^2 \partial_4 + x_1^2 x_2 x_3 \partial_1 \partial_3 ^2 \partial_4 - 2 x_1 x_2 x_4 \partial_3 \partial_4^2 - x_1^2 x_2 x_4 \partial_1 \partial_3 \partial_4^2\\
	&\Phi(f_{\varepsilon_1-\varepsilon_2}e_{\varepsilon_3-\varepsilon_4} e_{\varepsilon_1-\varepsilon_3}^2)=2 x_1 x_2 x_3 \partial_3 ^2 \partial_4 + x_1^2 x_2 x_3 \partial_1 \partial_3 ^2 \partial_4\\
	&\Phi(h_{1}f_{\varepsilon_3-\varepsilon_4}e_{\varepsilon_2-\varepsilon_4} e_{\varepsilon_1-\varepsilon_4})=x_1^2 x_2 x_4 \partial_1 \partial_3 \partial_4^2 - x_1 x_2^2 x_4 \partial_2 \partial_3 \partial_4^2\\
	&\Phi(f_{\varepsilon_3-\varepsilon_4}e_{\varepsilon_1-\varepsilon_2} e_{\varepsilon_2-\varepsilon_4}^2)=2 x_1 x_2 x_4 \partial_3 \partial_4^2 + x_1 x_2^2 x_4 \partial_2 \partial_3 \partial_4^2\\
	&\Phi(f_{\varepsilon_3-\varepsilon_4}f_{\varepsilon_1-\varepsilon_2} e_{\varepsilon_1-\varepsilon_4}^2)=2 x_1 x_2 x_4 \partial_3 \partial_4^2 + x_1^2 x_2 x_4 \partial_1 \partial_3 \partial_4^2.
	\end{align*}
	These relations imply that $\Phi(v')=0$, where $v'$ is given in Proposition $\ref{Zhu-algebra}$. Since $A(L_{-5/2} (sl(4))) = \mathcal{U} (sl(4)) /\langle v' \rangle $, the proof follows.
\end{proof}

\subsection{Some relaxed $L_{-5/2}( sl(5))$--modules}

In this subsection we present a realization of $L_{-5/2}( sl(4))$--modules $L_{-5/2}(t \omega_1)$ using $sl(n)$--modules $M(\bf a)$ defined in \cite{BL}, and the conformal embedding $gl(4)  \hookrightarrow sl(5)$ at level $k=-5/2$.

We recall the definition of modules $M(\bf a)$ from \cite{BL}. For ${\bf a}= (a_1, a_2, \dots, a_n) \in {\Bbb C}^n$, we define $M(\bf a)$ as a complex vector space spanned by 
$$\{ x^{\bf b} = x_1 ^{b_1} \cdots x_n ^{b_n} \ \vert \  b_i - a_i \in  {\Bbb Z} \ \forall i, \  \sum_{i=1} ^n  (b_i - a_i)  = 0 \}.$$
The action of $sl(n)$ is given by
\begin{equation} \label{sl(n)-action}
 e_{\varepsilon_i - \varepsilon_j} = x _i  \partial  _j, \  f_{\varepsilon_i - \varepsilon_j} = x _j  \partial  _i,
 \end{equation}
for $ 1 \leq i < j \leq n$.

\begin{proposition} \cite{BL}
	For each n-tuple ${\bf a}= (a_1,\ldots, a_n) \in \mathbb{C} ^n$  with $a_i \notin \mathbb{Z}$ for all $i$, $M(\bf a)$ is an irreducible, torsion-free $sl(n)$--module.
\end{proposition}

\begin{proposition}\label{M(a)}
	For each ${\bf a}= (a_1,\ldots, a_5) \in \mathbb{C} ^5$, $M(\bf a)$ is an $A(L_{-5/2} (sl(5)))$--module if and only if $a_1 + a_2 + a_3 + a_4 + a_5 = -5/2$.
\end{proposition}

\begin{proof}
	It follows from \cite{Pe-07} that the vector
	\begin{align*}
	u&=(\frac{3}{5} h_{\varepsilon_1-\varepsilon_2}(-1) e_{\theta}(-1) + \frac{1}{5} h_{\varepsilon_2-\varepsilon_3}(-1) e_{\theta}(-1) - \frac{1}{5} h_{\varepsilon_3-\varepsilon_4}(-1) e_{\theta}(-1) - \frac{3}{5} h_{\varepsilon_4-\varepsilon_5}(-1) e_{\theta}(-1) +\\ 
	&{  +}  \left( e_{\varepsilon_1-\varepsilon_2}(-1) e_{\varepsilon_2-\varepsilon_5}(-1) + e_{\varepsilon_1-\varepsilon_3}(-1) e_{\varepsilon_3-\varepsilon_5}(-1) + e_{\varepsilon_1-\varepsilon_4}(-1) e_{\varepsilon_4-\varepsilon_5}(-1) \right) - \frac{3}{2}e_{\theta}(-2))\mathbf{1}
	\end{align*}
	is a singular vector in $V^{-5/2}(sl(5))$, which generates the maximal ideal.
	The projection of $u$ in Zhu's algebra is a vector $u'= F([u])$ given by the following formula:
	\begin{align} \label{v'-sl5}
	u'&=\frac{3}{5} e_{\varepsilon_1-\varepsilon_5} h_{\varepsilon_1-\varepsilon_2} +  \frac{1}{5} e_{\varepsilon_1-\varepsilon_5} h_{\varepsilon_2-\varepsilon_3} - \frac{1}{5} e_{\varepsilon_1-\varepsilon_5} h_{\varepsilon_3-\varepsilon_4} - \frac{3}{5} e_{\varepsilon_1-\varepsilon_5} h_{\varepsilon_4-\varepsilon_5} +  \\ & +\left(e_{\varepsilon_2-\varepsilon_5} e_{\varepsilon_1-\varepsilon_2} + e_{\varepsilon_3-\varepsilon_5} e_{\varepsilon_1-\varepsilon_3} + e_{\varepsilon_4-\varepsilon_5} e_{\varepsilon_1-\varepsilon_4}  \right)+ \frac{3}{2}e_{\varepsilon_1-\varepsilon_5} \in \mathcal{U}(sl(5)). \nonumber
	\end{align}
	Clearly, $M(\bf a)$ is an $A(L_{-5/2} (sl(5)))$--module if and only if $u'$ annihilates $M(\bf a)$. Using formula (\ref{v'-sl5}), and the definition of $sl(5)$--action 
	(\ref{sl(n)-action}), we obtain that $u'$ acts on $M(\bf a)$ as
	\begin{equation} \label{action-v'}
	 \frac{3}{5}  x_1 \partial_ 5  \left (     x_1 \partial  _1   +   x_2 \partial  _2 
	+   x_3 \partial  _3 
	+ x_4 \partial  _4 
	+  x_5 \partial  _5 + \frac{5}{2}  \right).  	\end{equation}
	Let $x^{\bf b} = x_1 ^{b_1} \cdots x_5 ^{b_5}$ be an arbitrary element of $M(\bf a)$. 
	Relation (\ref{action-v'}) now implies that
	$$ u'.x^{\bf b} = \frac{3}{5}  b_5 (b_ 1 + b_2 + b_3 + b_4 + b_5 + 5/2 )x^{\bf b'},$$
	where we denoted 
	$$x^{\bf b'} = x_1 ^{b_1+1}  x_2 ^{b_2}  x_3 ^{b_3}  x_4 ^{b_4} x_5 ^{b_5-1}.$$
	The claim of Proposition now follows from the relation $ \sum_{i=1} ^5  (b_i - a_i)  = 0 $.
\end{proof}

\begin{theorem}
	For any $t \in {\Bbb C} $, the irreducible  $L_{-5/2} (sl(4))$--module $L_{-5/2} (t \omega_1)$ is a subquotient of an $L_{-5/2} (sl(5))$--module $\widehat{ M({\bf a})}$ such that  $\widehat{ M({\bf a})}(0) \cong  M({\bf a})$, for  ${\bf a} = ( t, 0, 0, 0, -t - \frac{5}{2})$.	
\end{theorem}
\begin{proof}
	 Proposition \ref{M(a)} implies that $M({\bf a})$ is an $ A(L_{-5/2} (sl(5)))$--module for  ${\bf a} = ( t, 0, 0, 0, -t - \frac{5}{2})$, $t \in \mathbb{C}$.
	 The Proposition \ref{Zhu1} (2) now implies that there exists an $L_{-5/2} (sl(5))$--module $\widehat{ M({\bf a})}$ such that  $\widehat{ M({\bf a})}(0) \cong  M({\bf a})$.
	 Vector 
	 $$v_t := x_1 ^t x_5^{-t - \tfrac{5}{2}}$$
	 is a singular vector for $\widehat{sl(4)}$ of $sl(4)$--weight $t \omega_1$.
	 Since $L_{-5/2} (sl(4)) \otimes M_c(1)$ is conformally embedded in $L_{-5/2} (sl(5))$, we obtain that $ \widetilde L_{-5/2}(t \omega_1)=\mathcal{U}(\widehat{ sl(4)}).v_t$
	 is an $L_{-5/2}(sl(4))$--module. Clearly, this implies that $L_{-5/2}(t \omega_1)$
     is a subquotient of $\widehat{ M({\bf a})}$, since it is a quotient of 
     $\widetilde L_{-5/2}(t \omega_1)$.
	 \end{proof}
The next remark shows that there exist indecomposable $L_{-5/2}(sl(4))$--modules in the category~$\mathcal O$. 
\begin{remark} \label{remark-ind}
Consider the singular vector $v_t$ for $t\in {\Bbb Z}_{\ge 0}$. Then  we see that:
\begin{itemize}
\item $ \widetilde U_t:=  \mathcal{U}(sl(4)) v_{t} $ is an infinite-dimensional highest weight  $sl(4)$--module.
\item The highest weight of $ \widetilde U_t$ is $\mu_1(t) =t \omega_1$,  so it is dominant integral. Therefore, $ \widetilde U_t$ is an indecomposable $A( L_{-5/2}(sl(4)))$--module.
\item $\widetilde L_{-5/2}(t \omega_1)$ is indecomposable $L_{-5/2}(sl(4))$--module.
\end{itemize}
\end{remark}

 \section*{Appendix: OPE for $W^k(sl(4), f_{subreg})$}
 Recall that the vertex algebra $W^k(sl(4), f_{subreg})$, obtained by the quantum Hamiltonian reduction,  is generated by five fields $J, \bar L, W, G^+, G^-$ of conformal weights $1,2,3,1,3$. We choose the new   Virasoro field   
 $ L = \bar L - \partial J$ so that $G^{\pm}$ have conformal weight $2$.
Here we recall the OPE for $W^k(sl(4), f_{subreg})$ from the paper \cite{CL}.

 \begin{align*}
L(z)L(w) &\sim - \frac{(8+3k)(17+8k)}{2(4+k)}(z-w)^{-4} + 2L(w)(z-w)^{-2} + \partial L(w)(z-w)^{-1},\\
L(z)J(w) &\sim J(w)(z-w)^{-2} + \partial J(w)(z-w)^{-1},\\
L(z)W(w) &\sim 3W(w)(z-w)^{-2} + \partial W(w)(z-w)^{-1},\\
L(z)G^{\pm}(w) &\sim 2G^{\pm}(w)(z-w)^{-2} +\partial G^{\pm}(w)(z-w)^{-1},\\
J(z)J(w) &\sim (2+\frac{3k}{4})(z-w)^{-2},\\
J(z)G^{\pm}(w) &\sim \pm G^{\pm}(w)(z-w)^{-1},\\
\ & \\
W(z)G^{\pm}(w) &\sim \pm \frac{2(4+k)(7+3k)(16+5k)}{(8+3k)^2}G^{\pm}(w)(z-w)^{-3} \\
&+  \left( \pm \frac{3(4+k)(16+5k)}{2(8+3k)} \partial G^{\pm} - \frac{6(4+k)(16+5k)}{(8+3k)^2} : J G^{\pm}: \right) (w)(z-w)^{-2}\\
&+ \left( - \frac{8(3+k)(4+k)}{(2+k)(8+3k)} : J \partial G^{\pm} : - \frac{4(4+k)(16+15k+3k^2)}{(2+k)(8+3k)^2} : (\partial J) G^{\pm}: \right. \\
&\pm \frac{(3+k)(4+k)}{2+k} \partial ^2 G^{\pm} \mp \frac{2(4+k)^2}{(2+k)(8+3k)} :LG^{\pm}: \\
&\pm \left.  \frac{4(4+k)(16+5k)}{(2+k)(8+3k)^2} :JJG^{\pm}: \right) (w)(z-w)^{-1},\\
\ & \\
G^{+}(z)G^{-}(w)&\sim (2+k)(5+2k)(8+3k)(z-w)^{-4} + 4(2+k)(5+2k)J(w)(z-w)^{-3}\\
&+  \bigg( -(2+k)(4+k)L + 6(2+k):JJ: +2(2+k)(5+2k)\partial J \bigg) (w)(z-w)^{-2}\\
&+ \left( (k+2)W + \frac{8(2+k)(32+11k)}{3(8+3k)^2}:JJJ:- \frac{4(2+k)(4+k)}{8+3k}:LJ:+6(2+k):(\partial J)J: \right.\\
&-\frac{1}{2}(2+k)(4+k)\partial L + \left. \frac{4(2+k)(26+17k+3k^2)}{3(8+3k)}\partial^2 J \right)(w)(z-w)^{-1},\\
\ & \\
W(z)W(w) &\sim \frac{2(k+4)(2k+5)(3k+7)(5k+16)}{3k+8}(z-w)^{-6}+ \cdots.
\end{align*}
The remaining terms in the OPE of $W(z)W(w)$ can be found in the paper \cite{FS}.

\end{document}